\documentclass{article}

\usepackage [utf8]{inputenc}
\usepackage [T1] {fontenc}
\usepackage{xspace} 
\usepackage{graphicx} 
\usepackage[linktocpage]{hyperref} 
\usepackage{enumerate}

\usepackage[toc,page]{appendix} 

\usepackage{geometry}
\usepackage{caption} 
\usepackage{subcaption}

\newcommand{\R}{\mathbb{R}}
\newcommand{\E}{\mathbb{E}}

\newcommand{\N}{\mathbb{N}}
\newcommand{\trans}{^{\top}}

\newcommand{\bs}{\boldsymbol}

\title{A Comprehensive Bayesian Treatment of the Universal Kriging model with Matérn correlation kernels}
\author{Joseph Muré}

\usepackage{amsthm}
\usepackage{amsmath}
\usepackage{amssymb}

\usepackage{natbib}
\DeclareMathOperator{\Tr}{Tr}
\DeclareMathOperator{\V}{Var}
\DeclareMathOperator{\Cov}{Cov}

\newtheorem{thm}{Theorem}
\newtheorem{cor}[thm]{Corollary}
\newtheorem{prop}[thm]{Proposition}

\newtheorem{lem}[thm]{Lemma}
\newtheorem{defn}[thm]{Definition}
\newtheorem{hyp}{Assumption}

\theoremstyle{remark}
\newtheorem*{rmq}{Remark}

\let\lvert=|\let\rvert=|

\newcommand{\corr}{ \bs{ \Sigma }_{ \bs{ \theta } } }
\newcommand{\derivee}{ \left( \frac{\partial}{\partial \theta_i } \corr \right)}

\begin{document}

\maketitle

\begin{abstract}
The Gibbs reference posterior distribution provides an objective full-Bayesian solution to the problem of prediction of a stationary Gaussian process with Matérn anisotropic kernel. A full-Bayesian approach is possible, because the posterior distribution is expressed as the invariant distribution of a uniformly ergodic Markovian kernel for which we give an explicit expression. In this paper, we show that it is appropriate for the Universal Kriging framework, that is when an unknown function is added to the stationary Gaussian process. We give sufficient conditions for the existence and propriety of the Gibbs reference posterior that apply to a wide variety of practical cases and illustrate the method with several examples. Finally, simulations of Gaussian processes suggest that the Gibbs reference posterior has good frequentist properties in terms of coverage of prediction intervals.
\end{abstract}


\section{Introduction}

Gaussian Stochastic Processes (GaSP) offer a convenient way of expressing the uncertainty about the value of some real-valued quantity on a given spatial domain $\mathcal{D}$  \citep{Ste99} when said quantity is only observed on a finite set of points in $\mathcal{D}$. This is why Gaussian Process Regression is used as a supervised learning method \citep[chapter 2]{Ras06}, although it originally appeared in the geostatistical literature \citep{Mat60}. In this paper, we follow the geostatistical naming convention for this model: Kriging. 

In Simple Kriging, the Gaussian Process is assumed to have zero mean and be stationary, so its distribution can be characterized by a positive variance parameter $\sigma^2$ and by an autocorrelation function $K$. The Universal Kriging framework adds another parameter : a mean function $f$. If $f$ is known, then subtracting it from the process returns us to the Simple Kriging framework. Allowing for an unknown mean function provides greater flexibility in the modeling by enabling some degree of non-stationarity \citep[section 2.3.2]{SWN}.\medskip

In practice, the mean function $f$ is assumed to belong to a $p$-dimensional ($p \in \N$) vector space $\mathcal{F}_p$, which is specified by means of a basis $(f_1,...,f_p)$.
Being a linear combination of $f_1,...,f_p$, the mean function $f$ is then encoded by the vector of linear coefficients $\bs{\beta} = (\beta_1,...,\beta_p) \trans$ : $f = \beta_1 f_1 + ... + \beta_p f_p$.

Therefore, what separates the Universal Kriging framework from its Simple counterpart is the addition of the $p$-dimensional parameter $\bs{\beta}$. \medskip

Isotropic autocorrelation kernels are usually characterized through a scale parameter on the vector space spanned by $\mathcal{D}$. If $K_\theta: <\mathcal{D}> \mapsto [-1,1]$ is one such kernel, then $\forall \bs{x} \in <\mathcal{D}>$ $K_\theta(\bs{x}) = K_1( \bs{x} / \theta)$. $\theta$ is called the correlation length of the kernel.

However, the assumption that the correlation structure be isotropic is very strong, and is rarely appropriate in the context of computer experiments, where each point in the spatial domain $\mathcal{D}$ represents a set of possibly heterogeneous parameters. In such settings, anisotropic autocorrelations are used, and these require one correlation length $\theta_i$ for every dimension of $\mathcal{D}$. Let $\bs{\theta}$ denote the vector of the correlation lengths of every dimension.

So, assuming the autocorrelation function to be characterized by a vector of correlation lengths $\bs{\theta}$, we are faced with the inference problem of estimating $(\bs{\beta},\sigma^2,\bs{\theta})$. Unfortunately, even in the Simple Kriging framework where the mean function is wholly known, estimating $(\sigma^2,\bs{\theta})$ may be difficult \citep{KOH01}. Indeed, when few observation points are available -- and this is often the case in the context of emulation of Computer experiments -- the likelihood function may be quite flat \citep{LS05}. This is why, instead of hazarding a guess at the ``true'' value of the parameters, it seems reasonable to represent the uncertainty through a posterior distribution. \medskip

In \citet{Mur18}, an objective posterior distribution on $(\sigma^2,\bs{\theta})$ is proposed in the context of Simple Kriging. In this paper we address the more general framework of Universal Kriging in order to obtain a distribution on $(\bs{\beta},\sigma^2,\bs{\theta})$.
The developments in both articles are based on Bernardo's reference prior theory.
The idea to use this theory in the context of Kriging first appeared in \citet{BDOS01}, and then was successively extended by \citet{paulo05}, \citet{KP12},
\citet{RSH12},
\citet{RSS13} 
and \citet{Gu16}.

 To use it we first need to order the parameters \citep{Ber05}. Because our main goal is to maximize the predictive capacity of the model, we are unable to outright say which parameter we care about most. However, a few common sense observations help: first, in order to profit from the work done in the Simple Kriging case, we separate $\bs{\beta}$, which refers to the mean function, from $(\sigma^2,\bs{\theta})$, which yields the covariance structure. Within the latter, $\bs{\theta}$ should have the priority over $\sigma^2$, because while $\sigma^2$ can very easily be accurately estimated once $\bs{\theta}$ is known, the reverse is not true. The same consideration will make us prioritize $(\sigma^2,\bs{\theta})$ over $\bs{\beta}$, because while knowing $\bs{\beta}$ reduces the problem to the Simple Kriging case, knowing $(\sigma^2,\bs{\theta})$ reduces it to a much simpler regression problem. \medskip

In Section \ref{Sec:prior_gibbs_univ} we derive the 
reference posterior distribution on $(\bs{\beta},\sigma^2)$ and the corresponding predictive distribution at unobserved points, both conditional to the observed data and the correlation parameter $\bs{\theta}$.

In Section \ref{Sec:ref_prior_dim1}, we derive analytical formulas for the reference prior on $(\bs{\beta},\sigma^2,\bs{\theta})$ in the case where $\bs{\theta}$ is a one-dimensional parameter. The main difficulty is that because we use exact marginalization, we have to deal with improper likelihoods. Our approach to tackle this problem is, we believe, simpler than \citep{BDOS01}'s, but we show that both solutions amount to restricting the amount of available observation data.

In Section \ref{Sec:Gibbs_posterior}, we prove the main result of the paper: in the context of a Matérn anisotropic correlation kernel \citep{Mat86,HS93} -- see Appendix \ref{App:noyaux_Matérn} for precise definitions -- under a few conditions, the Gibbs reference posterior on a multidimensional $\bs{\theta}$ exists. Combined with the ``partial'' reference posterior distribution on $(\bs{\beta},\sigma^2)$ conditional to $\bs{\theta}$, it provides a proper objective posterior distribution on all parameters $(\bs{\beta},\sigma^2,\bs{\theta})$ given the observed data. 
It is significant that this proper objective posterior distribution is well defined for Matérn anisotropic correlation kernels, because this class of correlation kernels has remarkable properties (see  \citet{Ste99} or \citet[chapter 2]{Bac13}). Notably, it allows the user to specify the smoothness of the realizations of the Gaussian Process.

In Section \ref{Sec:comp_predictions}, we evaluate the predictive performance of the Universal Kriging model with the Gibbs reference posterior distribution both in the context of a well-specified model and when emulating deterministic functions. We compare the full-Bayesian approach relying on the Gibbs reference posterior with plug-in approaches, where the parameters are assumed to be equal to either the Maximum Likelihood Estimator (MLE) or the Maximum A Posteriori (MAP) estimator.

\section{Analytical treatment of the location-scale parameters \texorpdfstring{$\bs{\beta}$}{Lg} and \texorpdfstring{$\sigma^2$}{Lg}} \label{Sec:prior_gibbs_univ}

Suppose our design set contains $n$ observation points. $n$ must be greater than $p$, otherwise the model is not identifiable. Let $\bs{H}$ be the $n \times p$ matrix whose columns contain the values of the $p$ basis functions at the $n$ observation points. Let us assume that the rank of $\bs{H}$ is $p$, because if it were not, the model would also not be identifiable.

Let $\bs{y}$ be the vector of the $n$ observations. Then $\bs{y}$ is a Gaussian vector and its distribution is 

\begin{equation} \label{Eq:loi_kri_univ}
\bs{y} | \bs{\beta}, \sigma^2, \bs{\theta} \sim \mathcal{N}(\bs{H \beta}, \sigma^2 \corr ),
\end{equation}

where $\corr$ is a correlation matrix that only depends on the design set and on the vector of correlation lengths $\bs{\theta}$. \medskip

In terms of likelihood, we have

\begin{equation} \label{Eq:vraisemblance}
L( \bs{y} \; | \;  \bs{\beta}, \sigma^2 , \bs{ \theta } ) = 
\left( \frac{ 1 }{ 2 \pi \sigma^2 } \right) ^ { \frac{n}{2} } | \bs{ \Sigma }_{ \bs{ \theta } } | ^ {- \frac{1}{2} } \exp \left\{ - \frac{ 1 }{ 2 \sigma^2 } \left(\bs{y} - \bs{H \beta} \right) \trans \bs{ \Sigma }_{ \bs{ \theta } }^{-1} \left( \bs{y} - \bs{H \beta} \right) \right\} \; .
\end{equation}

The aim of this section is to get the parameters $\bs{\beta}$ and $\sigma^2$ out of the way in order to focus on the more interesting parameter $\bs{\theta}$. For now, assume that $\bs{\theta}$ is known, which is to say that the correlation function is completely known. \medskip

\subsection{Reference prior and integrated likelihood when \texorpdfstring{$\bs{\theta}$}{Lg} is known.} \label{Subsec:prior_gibbs_univ_calcul}

Clearly, $\bs{\beta}$ is a location parameter and $\sigma:= \sqrt{\sigma^2}$ is a scale parameter for this model. Therefore, the joint reference prior is $\pi(\bs{\beta},\sigma^2 | \bs{\theta}) \propto 1 / \sigma^2$ regardless of the order of the parameters $(\bs{\beta},\sigma^2)$. \medskip

We now derive the posterior distributions $\pi(\bs{\beta} | \bs{y}, \sigma^2, \bs{\theta})$ and $\pi(\sigma^2| \bs{y}, \bs{\theta})$ as well as the integrated likelihoods $L^0(\bs{y} | \sigma^2, \bs{\theta}) := \int L(\bs{y} | \bs{\beta}, \sigma^2, \bs{\theta}) d \bs{\beta} $ and $L^1(\bs{y} | \bs{\theta}) := \iint L(\bs{y} | \bs{\beta}, \sigma^2, \bs{\theta}) / \sigma^2 d \bs{\beta} d \sigma^2$. \medskip

Gaussian theory makes it convenient to split $\bs{y}$ into two components : one that belongs to the subspace of $\R^n$ spanned by $\bs{H}$, and one that is orthogonal to the subspace spanned by $\bs{H}$. In order not to have to deal with degenerate Gaussian vectors, we define an $n \times p$ matrix $\bs{P}$ with full rank which spans the same subspace as $\bs{H}$ (Actually, for the time being, we may as well set $\bs{P} = \bs{H}$.) and an $n \times (n-p)$ matrix $\bs{W}$ with full rank which spans its orthogonal space. Thus $\bs{W} \trans \bs{H} = \bs{W} \trans \bs{P} = \bs{0}_{n-p,p}$ and  \medskip

\begin{align} 
\label{Eq:yOrthogonalH}\bs{W} \trans \bs{y} | \sigma^2, \bs{\theta} \sim \mathcal{N}(& \bs{0}_{n-p}, \sigma^2 \bs{W} \trans \corr \bs{W} ) \; ;\\
\label{Eq:ySelonH_sachant_beta} \bs{P} \trans \bs{y} | \bs{\beta}, \sigma^2, \bs{\theta}, \bs{W} \trans \bs{y}  \sim
 \mathcal{N}(& \bs{P} \trans \bs{H \beta} + \bs{P} \trans \corr \bs{W} \left( \bs{W} \trans \corr \bs{W} \right)^{-1} \bs{W} \trans \bs{y}, \\
   &\sigma^2 \bs{P} \trans \corr \bs{P} - \sigma^2 \bs{P} \trans \corr \bs{W} \left( \bs{W} \trans \corr \bs{W} \right)^{-1} \bs{W} \trans \corr \bs{P}). \nonumber
\end{align}

$\bs{\beta}$ having flat prior density, $\bs{P} \trans \bs{y} - \bs{P} \trans \bs{H \beta}$ has the same distribution whether $\bs{\beta}$, $\sigma^2$ and $\bs{\theta}$ or whether $\bs{P} \trans \bs{y} $, $\sigma^2$ and $\bs{\theta}$ are known.
Therefore, the posterior distribution of $\bs{P} \trans \bs{H \beta}$ if $\sigma^2$ and $\bs{\theta}$ are known is : 

\begin{equation} 
\begin{split}
\bs{P} \trans \bs{H \beta} | \sigma^2, \bs{\theta}, \bs{W} \trans \bs{y}, \bs{P} \trans \bs{y}  \sim
 \mathcal{N}(& \bs{P} \trans \bs{y} - \bs{P} \trans \corr \bs{W} \left( \bs{W} \trans \corr \bs{W} \right)^{-1} \bs{W} \trans \bs{y}, \\
  & \sigma^2 \bs{P} \trans \corr \bs{P} - \sigma^2 \bs{P} \trans \corr \bs{W} \left( \bs{W} \trans \corr \bs{W} \right)^{-1} \bs{W} \trans \corr \bs{P}) .
\end{split}  
\end{equation}

From there, we get the posterior distribution of $\bs{\beta}$ if $\sigma^2$ and $\bs{\theta}$ are known :

\begin{equation} \label{Eq:beta_sachant_ySelonH}
\begin{split}
\bs{\beta} &| \sigma^2, \bs{\theta}, \bs{y}  \sim 
 \mathcal{N}( (\bs{P} \trans \bs{H})^{-1} \bs{P} \trans \bs{y} - (\bs{P} \trans \bs{H})^{-1} \bs{P} \trans \corr \bs{W} \left( \bs{W} \trans \corr \bs{W} \right)^{-1} \bs{W} \trans \bs{y}, \\
  & \sigma^2 (\bs{P} \trans \bs{H})^{-1} \bs{P} \trans \corr \bs{P} (\bs{H} \trans \bs{P})^{-1} - \sigma^2 (\bs{P} \trans \bs{H})^{-1} \bs{P} \trans \corr \bs{W} \left( \bs{W} \trans \corr \bs{W} \right)^{-1} \bs{W} \trans \corr \bs{P} (\bs{H} \trans \bs{P})^{-1}) 
\end{split}  
\end{equation}

Moreover, (\ref{Eq:ySelonH_sachant_beta}) implies that the integrated likelihood of $\bs{P} \trans \bs{y}$, i.e. its likelihood averaged over the Lebesgue measure (the prior distribution on $\bs{\beta}$), is $|\bs{P} \trans \bs{H}|^{-1}$, where $| \cdot |$ denotes the absolute value of the determinant.

\begin{equation} \label{Eq:Py_ignorant_beta}
\begin{split}
\bs{P} \trans \bs{y} | \sigma^2, \bs{\theta}, \bs{W} \trans \bs{y}  \sim
\mathrm{Improper} \; \mathrm{``uniform"} \;  \mathrm{distribution} \; \mathrm{on} \; \R^p .
\end{split}  
\end{equation}

This means that if $\bs{\beta}$ is unknown, then $\bs{P} \trans \bs{y}$ can yield no information about $\sigma^2$ and $\bs{\theta}$. When $\bs{\beta}$ is unknown, all information about $\sigma^2$ and $\bs{\theta}$ is carried by $\bs{W} \trans \bs{y}$, because as is shown by (\ref{Eq:yOrthogonalH}), the predictive distribution on $\bs{W} \trans \bs{y}$ knowing $\sigma^2$ and $\bs{\theta}$ does not depend on $\bs{\beta}$. \medskip

A straightforward calculation yields that the posterior distribution of $\sigma^2$ is Inverse-Gamma :

\begin{equation} 
\begin{split}
\sigma^2 | \bs{\theta},  \bs{y}  \sim
\mathcal{IG} (\mathrm{shape} = (n-p)/2, \mathrm{rate} = \bs{y} \trans \bs{W} \left( \bs{W} \trans \corr \bs{W} \right)^{-1} \bs{W} \trans \bs{y} / 2).
\end{split}  
\end{equation}

The posterior distribution on $\sigma^2$ (knowing $\bs{\theta}$) does not take into account $\bs{P} \trans \bs{y}$, because all information contained in $\bs{P} \trans \bs{y}$ is given in the posterior distribution of $\bs{\beta}$ conditional to $\sigma^2$ and $\bs{\theta}$.

We conclude this subsection with the formulas for the likelihoods with the parameters $\bs{\beta}$ and $\sigma^2$ successively integrated out.

\begin{align}
L^0(\bs{y} | \sigma^2, \bs{\theta}) = \int L(\bs{y} | \bs{\beta}, \sigma^2, \bs{\theta}) d \bs{\beta} 
&= \left( \frac{ 1 }{ 2 \pi \sigma^2 } \right) ^ { \frac{n-p}{2} } | \bs{W} \trans \bs{ \Sigma }_{ \bs{ \theta } } \bs{W} | ^ {- \frac{1}{2} } \exp \left\{ - \frac{ 1 }{ 2 \sigma^2 } \bs{y} \trans \bs{W}  \left( \bs{W} \trans \bs{ \Sigma }_{ \bs{ \theta } } \bs{W} \right)^{-1} \bs{W} \trans \bs{y} \right\} \; ;\\
L^1(\bs{y} | \bs{\theta}) = \int L^0(\bs{y} | \sigma^2, \bs{\theta}) / \sigma^2 d \sigma^2
 &= \left( \frac{2 \pi^{n-p}}{\Gamma\left(\frac{n-p}{2}\right)} \right)^{-1}  
 | \bs{W} \trans \bs{ \Sigma }_{ \bs{ \theta } } \bs{W} | ^ {- \frac{1}{2} }
  \left( \bs{y} \trans \bs{W} \left( \bs{W} \trans \bs{ \Sigma }_{ \bs{ \theta } } \bs{W} \right)^{-1} \bs{W} \trans \bs{y} \right)^{-\frac{n-p}{2} }. \label{Eq:vraisemblance_integree}
\end{align}

\subsection{Posterior predictive distribution when \texorpdfstring{$\bs{\theta}$}{Lg} is known.}

\newcommand{\HNouveauNouveau}{\bs{H}_{0,0}}
\newcommand{\HNouveauAncien}{\bs{H}_{0,\cdot}}
\newcommand{\HAncienNouveau}{\bs{H}_{\cdot,0}}
\newcommand{\corrNouveauNouveau}{\bs{\Sigma}_{\bs{\theta},0,0}}
\newcommand{\corrNouveauAncien}{\bs{\Sigma}_{\bs{\theta},0,\cdot}}
\newcommand{\corrAncienNouveau}{\bs{\Sigma}_{\bs{\theta},\cdot,0}}

\newcommand{\Emarg}{\bs{E}_{0}}
\newcommand{\EmargW}{\bs{E}_{\bs{\theta},0}^{\bs{W}}}
\newcommand{\Vmarg}{\bs{S}_{\bs{\theta},0,0}}
\newcommand{\VmargW}{\bs{S}_{\bs{\theta},0,0}^{\bs{W}}}
\newcommand{\corrmargAncienNouveau}{\bs{S}_{\bs{\theta},\cdot,0}}
\newcommand{\corrmargNouveauAncien}{\bs{S}_{\bs{\theta},0,\cdot}}
\newcommand{\corrmargWNouveauAncien}{\bs{S}_{\bs{\theta},0,\bs{W}}}
\newcommand{\corrmargWAncienNouveau}{\bs{S}_{\bs{\theta},\bs{W},0}}

Following \cite{SWN} (Theorem 4.1.2., case (4)),
 we derive conditionally to $\bs{\theta}$ the posterior predictive distribution of the values 
 taken by the process at unobserved points.

In order to simplify notations in this subsection, all the distributions we consider are, until further notice, conditional to $\sigma^2$ and $\bs{\theta}$ even with no explicit mention.
Equation (\ref{Eq:loi_kri_univ}) can be usefully restated in the following way :

\begin{equation}
\left.
\begin{pmatrix}
\bs{P} \trans \bs{y} - \bs{P} \trans \bs{H} \bs{\beta}  \\ \bs{W} \trans \bs{y}
\end{pmatrix}
\right| \bs{P} \trans \bs{H} \bs{\beta}
\sim
\mathcal{N} \left( \bs{0}_n, \sigma^2 
\begin{pmatrix}
\bs{P} \trans  \\ \bs{W} \trans
\end{pmatrix}
\corr
\begin{pmatrix}
\bs{P}  & \bs{W}
\end{pmatrix}
\right).
\end{equation}

Because the prior distribution on $\bs{\beta}$ is flat,
$
\begin{pmatrix}
\bs{P} \trans \bs{y} - \bs{P} \trans \bs{H} \bs{\beta}  \\ \bs{W} \trans \bs{y}
\end{pmatrix}
$
and its opposite have the same distribution when conditional respectively to $\bs{P} \trans \bs{H} \bs{\beta}$ and $\bs{P} \trans \bs{y}$.

\begin{equation} \label{eq:Ancien}
\left.
\begin{pmatrix}
\bs{P} \trans \bs{H} \bs{\beta} - \bs{P} \trans \bs{y} \\ - \bs{W} \trans \bs{y}
\end{pmatrix}
\right| \bs{P} \trans \bs{y}
\sim
\mathcal{N} \left( \bs{0}_n, \sigma^2 
\begin{pmatrix}
\bs{P} \trans  \\ \bs{W} \trans
\end{pmatrix}
\corr
\begin{pmatrix}
\bs{P}  & \bs{W}
\end{pmatrix}
\right)
\end{equation}

Let $\bs{y}_0$ be the values of the Gaussian Process at the $n_0$ unobserved points. We denote $\HNouveauNouveau$ the $n_0 \times p$ matrix whose columns contain the values of the $p$  basis functions at the unobserved points, $\corrNouveauNouveau$ the $n_0 \times n_0$ correlation matrix of $\bs{y}_0$, $\corrNouveauAncien$ the $n_0 \times n$ correlation matrix between $\bs{y}_0$ and $\bs{y}$ and $\corrAncienNouveau$ its transpose. It is also convenient to define the $n_0 \times n$ matrix $\HNouveauAncien = \HNouveauNouveau \left( \bs{P} \trans \bs{H} \right)^{-1}$ and its transpose $\HAncienNouveau$. With these notations, the distribution  of $\bs{y}_0$ when $\bs{y}$ and $\bs{\beta}$ are known is 

\begin{equation} \label{eq:NouveauSachantAncien}
\bs{y}_0 | \bs{\beta}, \bs{y} \sim \mathcal{N} \left( \HNouveauNouveau \bs{\beta} + \corrNouveauAncien \corr^{-1} (\bs{y} - \bs{H \beta}) , \corrNouveauNouveau - \corrNouveauAncien \corr^{-1} \corrAncienNouveau \right)
\end{equation}

Now, the distribution of $\bs{y}_0$ when $\bs{y}$ and $\bs{\beta}$ are known and the distribution of
$
\begin{pmatrix}
\bs{P} \trans \bs{H} \bs{\beta} - \bs{P} \trans \bs{y} \\ - \bs{W} \trans \bs{y}
\end{pmatrix}
$
when $\bs{P} \trans \bs{y}$ is known jointly define some probability distribution on the vector $(\bs{y}_0, \bs{P} \trans \bs{H \beta} - \bs{P} \trans \bs{y}, - \bs{W} \trans \bs{y} ) \trans $.

This distribution is given in the following proposition. In order to give it a concise expression, it is convenient to require that $\bs{PP} \trans + \bs{WW} \trans = \bs{I}_n$, which simply means that the columns of $\bs{P}$ and $\bs{W}$ form an orthonormal basis of $<\bs{H}>$ and its orthogonal space respectively.

\begin{prop} \label{Prop:grand_vect_gauss}
Assume that $\bs{PP} \trans + \bs{WW} \trans = \bs{I}_n$. Then the probability distribution on the vector of $\R^{n_0+n}$ $(\bs{y}_0, \bs{P} \trans \bs{H \beta} - \bs{P} \trans \bs{y}, - \bs{W} \trans \bs{y} ) \trans $
conditional to $\bs{P} \trans \bs{y}$
is the following multivariate normal distribution :

\begin{equation}
\mathcal{N} \left(
\begin{pmatrix} \Emarg \\ \bs{0}_n \end{pmatrix}
,
\sigma^2
\begin{pmatrix}
\Vmarg  &  \corrmargNouveauAncien  \\ \corrmargAncienNouveau  &
\begin{pmatrix}
\bs{P} \trans  \\ \bs{W} \trans
\end{pmatrix}
\corr
\begin{pmatrix}
\bs{P}  & \bs{W}
\end{pmatrix}
\end{pmatrix}
\right).
\end{equation}

We use the following notations :

\begin{align}
\Emarg :&= \HNouveauAncien \bs{P} \trans \bs{y} \nonumber \\ 
\Vmarg 
:&=
\corrNouveauNouveau +
\HNouveauAncien \bs{P} \trans \corr \bs{P} \HAncienNouveau
- \HNouveauAncien \bs{P} \trans \corrAncienNouveau - \corrNouveauAncien \bs{P} \HAncienNouveau
\nonumber \\
\corrmargNouveauAncien  
\begin{pmatrix}
\bs{P} \trans  \\ \bs{W} \trans
\end{pmatrix}
:&=
\HNouveauAncien \bs{P} \trans \corr - \corrNouveauAncien
\nonumber \\
\corrmargAncienNouveau :&= \corrmargNouveauAncien \trans
\nonumber 
\end{align}
\end{prop}

\begin{proof}

First, notice that the mean vector of the Normal distribution given by Equation \ref{eq:NouveauSachantAncien} can be rewritten as

\begin{equation}
(\HNouveauAncien - \corrNouveauAncien \corr^{-1} \bs{P} )\bs{P} \trans \bs{H \beta} + \corrNouveauAncien \corr^{-1} \bs{WW} \trans \bs{y} + \corrNouveauAncien \corr^{-1} \bs{PP} \trans \bs{y},
\end{equation}

which is a linear mapping of the vector $(\bs{P} \trans \bs{H \beta}, \bs{W} \trans \bs{y}, \bs{P} \trans \bs{y}) \trans $. Now, Equation \ref{eq:Ancien} tells us that conditional to $\bs{P} \trans \bs{y}$, $(\bs{P} \trans \bs{H \beta}, \bs{W} \trans \bs{y}, \bs{P} \trans \bs{y}) \trans $ is a (degenerate) Gaussian vector, so Gaussian theory implies that conditional to $\bs{P} \trans \bs{y}$, $(\bs{y}_0, \bs{P} \trans \bs{H \beta}, \bs{W} \trans \bs{y}, \bs{P} \trans \bs{y}) \trans $ is a Gaussian vector and therefore $(\bs{y}_0, \bs{P} \trans \bs{H \beta} - \bs{P} \trans \bs{y}, - \bs{W} \trans \bs{y}) \trans $ is one as well. So 
all that remains to be shown is that its mean and covariance are those given by Proposition \ref{Prop:grand_vect_gauss}.

To do this, we compute
$\bs{E}_{\bs{\theta}}^{\bs{y},\bs{\beta}}$ and $\sigma^2 \bs{S}_{\bs{\theta}}^{\bs{y},\bs{\beta}}$, the conditional mean and variance of $\bs{y}_0$ given $\bs{W} \trans \bs{y}$, $\bs{P} \trans \bs{y}$ and $\bs{P} \trans \bs{H} \bs{\beta}$ and check that they fit the parameters of (\ref{eq:NouveauSachantAncien}).

\begin{align}
\bs{E}_{\bs{\theta}}^{\bs{y},\bs{\beta}} &=
\Emarg
+ \corrmargNouveauAncien  
\begin{pmatrix}
\bs{P} \trans  \\ \bs{W} \trans
\end{pmatrix}
\corr^{-1}
\begin{pmatrix}
\bs{P}  & \bs{W}
\end{pmatrix}
\begin{pmatrix}
\bs{P} \trans \bs{H} \bs{\beta} - \bs{P} \trans \bs{y} \\ - \bs{W} \trans \bs{y}
\end{pmatrix}
=
\HNouveauNouveau \bs{\beta} + \corrNouveauAncien \corr^{-1} (\bs{y} - \bs{H} \bs{\beta}) \\
\bs{S}_{\bs{\theta}}^{\bs{y},\bs{\beta}} &=
\Vmarg 
- \corrmargAncienNouveau \trans 
\begin{pmatrix}
\bs{P} \trans  \\ \bs{W} \trans
\end{pmatrix}
\corr^{-1}
\begin{pmatrix}
\bs{P}  & \bs{W}
\end{pmatrix}
\corrmargAncienNouveau 
=
\corrNouveauNouveau - \corrNouveauAncien \corr^{-1} \corrAncienNouveau 
\end{align}

\end{proof}

From this point onwards, distributions are no longer implicitly conditional to $\sigma^2$ and $\bs{\theta}$.

\begin{cor}
Assume that $\bs{PP} \trans + \bs{WW} \trans = \bs{I}_n$. The predictive distribution when $\bs{\beta}$ is unknown -- i.e. the distribution of $\bs{y}_0$ conditional to $\bs{y}$, $\sigma^2$ and $\bs{\theta}$ -- is Normal. With the notations of Proposition \ref{Prop:grand_vect_gauss}, it has mean vector 
$\Emarg - \corrmargAncienNouveau \bs{W} \left( \bs{W} \trans \corr \bs{W} \right)^{-1} \bs{W} \trans \bs{y}$
and covariance matrix 
$$ \sigma^2 \left\{ \Vmarg - \corrmargAncienNouveau \bs{W} \left( \bs{W} \trans \corr \bs{W} \right)^{-1} \bs{W} \trans \corrmargNouveauAncien \right\}.$$
\end{cor}

\begin{cor}
Assume that $\bs{PP} \trans + \bs{WW} \trans = \bs{I}_n$. The predictive distribution when both $\bs{\beta}$ and $\sigma^2$ are unknown -- i.e. the distribution of $\bs{y}_0$ conditional to $\bs{y}$ and $\bs{\theta}$ -- is multivariate Student with $n-p$ degrees of freedom. With the notations of Proposition \ref{Prop:grand_vect_gauss}, it has location vector 
$\Emarg - \corrmargAncienNouveau \bs{W} \left( \bs{W} \trans \corr \bs{W} \right)^{-1} \bs{W} \trans \bs{y}$
and scale matrix 
$$ \frac{\bs{y} \trans \bs{W} \left( \bs{W} \trans \corr \bs{W} \right)^{-1} \bs{W} \trans \bs{y}}{n-p}\left\{ \Vmarg - \corrmargAncienNouveau \bs{W} \left( \bs{W} \trans \corr \bs{W} \right)^{-1} \bs{W} \trans \corrmargNouveauAncien \right\}.$$
\end{cor}

\section{Reference prior on a one-dimensional \texorpdfstring{$\bs{\theta}$}{Lg} } \label{Sec:ref_prior_dim1}

In this section, $\bs{\theta}$ is assumed to be a scalar parameter, which we emphasize by denoting it $\theta$.

\renewcommand{\corr}{ \bs{ \Sigma }_{ \theta } }
\renewcommand{\derivee}{ \left( \frac{\partial}{\partial \theta } \corr \right)}

Because of (\ref{Eq:yOrthogonalH}) and (\ref{Eq:Py_ignorant_beta}), it is fairly obvious that
the reference prior on $\theta$ is the same as in the Simple Kriging case \citep{Mur18}, but with $\bs{y}$ being replaced by $\bs{W} \trans \bs{y}$ and $\bs{ \Sigma }_{ \theta }$ by $\bs{W} \trans \bs{ \Sigma }_{ \theta }\bs{W}$. Naturally, because $\bs{W} \trans \bs{y} \in \R^{n-p}$ instead of $\R^n$, $n$ should also be replaced by $n-p$ :

\begin{prop}
The reference prior on $\theta$ is:
\begin{equation} \label{Prior1D}
\pi( \theta ) \propto \sqrt{ \Tr \left[ 
\left\{ \bs{W} \trans \frac{ \partial }{\partial \theta } ( \bs{ \Sigma }_{ \theta } )  \bs{W}
 \left( \bs{W} \trans \bs{ \Sigma }_{ \theta } \bs{W} \right)^{-1} \right\}^2 \right] - \frac{1}{ n-p } \left[ \Tr \left\{ \bs{W} \trans \frac{ \partial }{\partial \theta } ( \bs{ \Sigma }_{ \theta } ) \bs{W} \left( \bs{W} \trans \bs{ \Sigma }_{ \theta } \bs{W} \right)^{-1} \right\}\right]^2 }.
\end{equation}

\end{prop}

We now prove that this result is in keeping with the previous work of \citep{BDOS01}:

\begin{prop}
The reference prior on $\theta$ can also be written as:
\begin{equation} \label{Eq:Gibbs_ref_prior_theta_berger}
\pi(\theta) \propto 
\sqrt{\Tr \left[ \left\{ \frac{ \partial }{\partial \theta } ( \bs{ \Sigma }_{ \theta } )  \bs{ \Sigma }_{ \theta }^{-1} \bs{Q}_\theta \right\}^2 \right] 
-
\frac{1}{n-p} \left[ \Tr \left\{\frac{ \partial }{\partial \theta } ( \bs{ \Sigma }_{ \theta } )  \bs{ \Sigma }_{ \theta }^{-1} \bs{Q}_\theta \right\}\right]^2 },
\end{equation}
where $\bs{Q}_\theta := \bs{I}_n - \bs{H} \left( \bs{H} \trans \corr^{-1} \bs{H} \right)^{-1} \bs{H} \trans \corr^{-1}$.
\end{prop}

\begin{proof}
\cite{BDOS01} describe an alternative method to the one described in Subsection \ref{Subsec:prior_gibbs_univ_calcul} for the computation of the reference prior. Denoting $L(\bs{y} | \bs{\beta}, \sigma^2, \theta)$ the likelihood of the model, i.e. the density of the probability distribution of $\bs{y}$ when $\bs{\beta}$, $\sigma^2$ and $\theta$ are known, they compute $L^1(\bs{y} | \theta) = \iint L(\bs{y} | \bs{\beta}, \sigma^2, \theta) / \sigma^2 d \bs{\beta} d \sigma^2 \propto | \corr |^{-1/2} | \bs{H} \trans \corr^{-1} \bs{H} |^{-1/2} \exp \left[ - \left(2 \sigma^2 \right)^{-1} \bs{y} \trans \bs{Q}_\theta \trans \corr^{-1} \bs{Q}_\theta \bs{y} \right]$.
$\bs{Q}_\theta$ is the orthogonal projection on the orthogonal of the subspace of $\R^n$ spanned by $\bs{H}$, where orthogonality is defined by the scalar product $(\bs{a},\bs{b}) \mapsto \bs{a} \trans \bs{ \Sigma }_{ \theta }^{-1} \bs{b}$. This implies that $\bs{Q}_\theta \trans \bs{ \Sigma }_{ \theta }^{-1} \bs{Q}_\theta = \bs{ \Sigma }_{ \theta }^{-1} \bs{Q}_\theta$.  \medskip

First, we must compute 
the variance of the derivative of $L^1(\bs{y} | \bs{\theta})$ with respect to $\theta$.

\begin{equation}
\partial_{\theta} L^1(\bs{y} | \theta)
= - \frac{n}{2} 
\frac{\bs{y} \trans \partial_{\theta} \left\{\bs{Q}_\theta \trans \corr^{-1} \bs{Q}_\theta \right\} \bs{y}}
{\bs{y} \trans \bs{Q}_\theta \trans \corr^{-1} \bs{Q}_\theta \bs{y}} + C_{\theta} ,
\end{equation}

where $C_{\theta}$ is some additive constant.

\begin{equation}
\begin{split}
\partial_{\theta} \left\{\bs{Q}_\theta \trans \corr^{-1} \bs{Q}_\theta \right\}
&= \partial_{\theta} \left\{ \corr^{-1} - \corr^{-1} \bs{H} \left( \bs{H} \trans \corr^{-1} \bs{H} \right)^{-1} \bs{H} \trans \corr^{-1} \right\} \\
&= - \corr^{-1} \left(\partial_{\theta} \corr\right) \corr^{-1}
+ \corr^{-1} \left(\partial_{\theta} \corr\right) \corr^{-1} \bs{H} \left(\bs{H} \trans \corr^{-1} \bs{H}\right)^{-1} \bs{H} \trans \corr^{-1} \\
& \qquad - \corr^{-1} \bs{H} \left( \bs{H} \trans \corr^{-1} \bs{H}\right)^{-1} \bs{H} \trans \corr^{-1} \left(\partial_{\theta} \corr\right) \corr^{-1} \bs{H} \left( \bs{H} \trans \corr^{-1} \bs{H}\right)^{-1} \bs{H} \trans \corr^{-1} \\
& \qquad + \corr^{-1} \bs{H} \left( \bs{H} \trans \corr^{-1} \bs{H}\right)^{-1} \bs{H} \trans \corr^{-1} \left(\partial_{\theta} \corr\right) \corr^{-1} \\
&= - \corr^{-1} \left(\partial_{\theta} \corr\right) \corr^{-1} \bs{Q}_\theta
+ \corr^{-1} \bs{H} \left( \bs{H} \trans \corr^{-1} \bs{H}\right)^{-1} \bs{H} \trans \corr^{-1} \left(\partial_{\theta} \corr\right) \corr^{-1} \bs{Q}_\theta \\
&= - \corr^{-1} \bs{Q}_\theta \left(\partial_{\theta} \corr\right) \corr^{-1} \bs{Q}_\theta
= - \bs{Q}_\theta \trans \corr^{-1} \left(\partial_{\theta} \corr\right) \corr^{-1} \bs{Q}_\theta.
\end{split}
\end{equation}

The last step in the above computation is due to the fact that as $\corr^{-1} \bs{Q}_\theta = \bs{Q}_\theta \trans \corr^{-1} \bs{Q}_\theta$, it is symmetric. Finally, we obtain

\begin{equation}
\partial_{\theta} L^1(\bs{y} | \bs{\theta})
= \frac{n}{2} 
\frac{\bs{y} \trans \bs{Q}_\theta \trans \corr^{-1} \left(\partial_{\theta} \corr\right) \corr^{-1} \bs{Q}_\theta \bs{y}}
{\bs{y} \trans \bs{Q}_\theta \trans \corr^{-1} \bs{Q}_\theta \bs{y}} + C_{\theta} .
\end{equation}

Define a matrix $ \sqrt{ \bs{ \Sigma }_{  \theta } } $ such that $ \bs{ \Sigma }_{ \theta } = \sqrt{ \bs{ \Sigma }_{  \theta } } \sqrt{ \bs{ \Sigma }_{  \theta } } \trans $. 

Let $f_{\theta} : \R^n \rightarrow \R^n$ be defined by $f_{\theta}(\bs{y}) = \sqrt{ \bs{ \Sigma }_{  \theta } }^{-1}  \bs{Q}_\theta \bs{y} / \sqrt{\bs{y} \trans \bs{Q}_\theta \trans \corr^{-1} \bs{Q}_\theta \bs{y}} $.

\begin{equation}
\partial_{\theta} L^1(\bs{y} | \bs{\theta})
= \frac{n}{2} 
f_{\theta}(\bs{y}) \trans
\sqrt{ \bs{ \Sigma }_{  \theta } }^{-1}  \left(\partial_{\theta} \corr\right) \left(\sqrt{ \bs{ \Sigma }_{  \theta } }^{-1} \right) \trans
f_{\theta}(\bs{y})
 + C_{\theta} .
\end{equation}

$f_{\theta}$ pushes the probability distribution $\mathcal{N} (\bs{H \beta} , \sigma^2 \corr) $ onto the uniform distribution
 on the intersection of the unit sphere $S^{n-1}$ with the subspace of $\R^n$ spanned by $\sqrt{ \bs{ \Sigma }_{  \theta } }^{-1} \bs{Q}_\theta$, which is a sphere of dimension $n-p$.
 
\begin{lem} \label{Lem:variance_matrix}
If $\bs{U}$ is a random variable with uniform probability distribution on $S^{n-1}$, then for every real symmetric matrix $\bs{M}$, the variance of $\bs{U}  \trans \bs{M} \bs{U}$ is proportional to 
$
\Tr \left[ \bs{M}^2 \right] - n^{-1} \Tr \left[ \bs{M} \right]^2.
$
\end{lem}

The proof can be found in Appendix \ref{Sec:prior_gibbs_univ:App}.

\begin{cor} \label{Cor:quadratic_form}
Let $q$ be a quadratic form on $\R^n$ and let $q_1,...,q_n$ be its eigenvalues. Then, if $\bs{U}$ is a random variable with uniform probability distribution on $S^{n-1}$, the variance of the random variable $q(\bs{U})$ is proportional to $\sum_{i=1}^n q_i^2 - n^{-1} \left( \sum_{i=1}^n q_i \right)^2$.
\end{cor}

Let us define the inner product $< \bs{a} | \bs{b} >_{\theta} = \bs{a} \trans \sqrt{ \bs{ \Sigma }_{  \theta } }^{-1}  \left(\partial_{\theta} \corr\right) \left(\sqrt{ \bs{ \Sigma }_{  \theta } }^{-1} \right) \trans \bs{b}$.

Because we have
$
 \bs{Q}_\theta \trans \corr^{-1} \left(\partial_{\theta} \corr\right) \corr^{-1} \bs{Q}_\theta
 =
  \bs{Q}_\theta \trans \corr^{-1} \bs{Q}_\theta \left(\partial_{\theta} \corr\right) \bs{Q}_\theta \trans \corr^{-1} \bs{Q}_\theta
  $,
if $\bs{a}$ and $\bs{b}$ belong to the vector space spanned by $\sqrt{ \bs{ \Sigma }_{  \theta } }^{-1} \bs{Q}_\theta$, then   
  $< \bs{a} | \bs{b} >_{\theta} = \bs{a} \trans \sqrt{ \bs{ \Sigma }_{  \theta } }^{-1} 
  \bs{Q}_\theta \left(\partial_{\theta} \corr\right)
  \bs{Q}_\theta \trans
   \left(\sqrt{ \bs{ \Sigma }_{  \theta } }^{-1} \right) \trans \bs{b}$. 
Moreover, if either $\bs{a}$ or $\bs{b}$ is orthogonal to the vector space spanned by $\sqrt{ \bs{ \Sigma }_{  \theta } }^{-1} \bs{Q}_\theta$ in the sense of the usual scalar product, then 
$\bs{a} \trans \sqrt{ \bs{ \Sigma }_{  \theta } }^{-1} 
  \bs{Q}_\theta \left(\partial_{\theta} \corr\right)
  \bs{Q}_\theta \trans
   \left(\sqrt{ \bs{ \Sigma }_{  \theta } }^{-1} \right) \trans \bs{b} = 0$.

Therefore, the matrix $\bs{M}_\theta := \sqrt{ \bs{ \Sigma }_{  \theta } }^{-1} 
  \bs{Q}_\theta \left(\partial_{\theta} \corr\right)
  \bs{Q}_\theta \trans
   \left(\sqrt{ \bs{ \Sigma }_{  \theta } }^{-1} \right) \trans$
represents the restriction of the inner product $< \cdot | \cdot >_{\theta}$ to the vectorial subspace of $\R^n$ spanned by $\sqrt{ \bs{ \Sigma }_{  \theta } }^{-1} \bs{Q}_\theta$.

Combining this observation with Corollary \ref{Cor:quadratic_form} yields the following Corollary.
  
\begin{cor}
Define $\bs{M}_\theta := \sqrt{ \bs{ \Sigma }_{  \theta } }^{-1} 
  \bs{Q}_\theta \left(\partial_{\theta} \corr\right)
  \bs{Q}_\theta \trans
   \left(\sqrt{ \bs{ \Sigma }_{  \theta } }^{-1} \right) \trans$.
If $\bs{U}$ is a random variable with uniform probability distribution on the intersection of the unit sphere $S^{n-1}$ with the vectorial subspace of $\R^n$ of dimension $n-p$ spanned by $\sqrt{ \bs{ \Sigma }_{  \theta } }^{-1} \bs{Q}_\theta$,
then the variance of $\bs{U} \trans 
\sqrt{ \bs{ \Sigma }_{  \theta } }^{-1}  \left(\partial_{\theta} \corr\right) \left(\sqrt{ \bs{ \Sigma }_{  \theta } }^{-1} \right) \trans
\bs{U}$ is proportional to 
$
\Tr \left[ \bs{M}_\theta^2 \right] - (n-p)^{-1} \Tr \left[ \bs{M}_\theta \right]^2.
$

\end{cor}

From there, Equation (\ref{Eq:Gibbs_ref_prior_theta_berger}) follows trivially. \medskip

Both expressions of the Gibbs reference prior (\ref{Prior1D}) and (\ref{Eq:Gibbs_ref_prior_theta_berger}) are equal. This can be seen by noticing that the two methods yield two different expressions of

$$\int L(\bs{y} | \bs{\beta}, \sigma^2, \bs{\theta}) d \bs{\beta} \propto \exp \left[ - \left(2 \sigma^2 \right)^{-1} \bs{y} \trans \bs{W} \left(\bs{W} \trans \corr \bs{W} \right)^{-1} \bs{W} \trans \bs{y} \right] \propto \exp \left[ - \left(2 \sigma^2 \right)^{-1} \bs{y} \trans \bs{Q}_\theta \trans \corr^{-1} \bs{Q}_\theta \bs{y} \right].$$ If this is to hold for all $\bs{y} \in \R^n$, then we have the equality 

\begin{equation}
\bs{W} \left(\bs{W} \trans \corr \bs{W} \right)^{-1} \bs{W} \trans
=\bs{Q}_\theta \trans \corr^{-1} \bs{Q}_\theta
= \corr^{-1} \bs{Q}_\theta,
\end{equation}

which coupled with the properties of the trace implies that (\ref{Prior1D}) and (\ref{Eq:Gibbs_ref_prior_theta_berger}) are the same.

\end{proof}

\renewcommand{\corr}{ \bs{ \Sigma }_{ \bs{\theta} } }
\renewcommand{\derivee}{  \partial_{\theta_i} \corr }

\section{The Gibbs reference posterior on a multi-dimensional \texorpdfstring{$\bs{\theta}$}{Lg} } \label{Sec:Gibbs_posterior}

\subsection{Definition}

In the case of multidimensional $\bs{\theta}$, reference prior theory gives a choice between 1) considering $\bs{\theta}$ as a single parameter or 2) defining an ordering on the scalar parameters $\theta_1,...,\theta_r$. Both possibilities are unsatisfactory, albeit in different ways. Concerning 1), Jeffreys' prior is unsuited to dealing with multidimensional parameters \citep{RCR09} and besides, the posterior may be improper. Concerning 2), further integration of the likelihood (\ref{Eq:vraisemblance_integree}) would be analytically intractable, even if it were possible to define a non-arbitrary ordering of the coordinates of $\bs{\theta}$. \medskip

We propose a quasi-posterior distribution based on the reference posterior of models where only one coordinate of $\bs{\theta}$ is unknown. 
For any integer $i \in [\!|1,r]\!]$, we collectively denote $\bs{\theta}_{-i}$ all coordinates of $\bs{\theta}$ except the $i$-th: $\bs{\theta}_{-i} = (\theta_j)_{j \in [\![1,r]\!] \setminus \{i\}}$.

Consider now $\pi_i(\theta_i | \bs{\theta}_{-i})$ the reference prior distribution on $\theta_i$ conditional to $\bs{\theta}_{-i}$ and the associated reference posterior distribution
$
\pi_i(\theta_i | \bs{y}, \bs{\theta}_{-i}) \propto  L^1(\bs{y} | \bs{\theta}) \pi_i(\theta_i | \bs{\theta}_{-i}),
$.

The conditional reference prior $\pi_i(\theta_i | \bs{\theta}_{-i})$ is given by:
\begin{equation} \label{Eq:Gibbs_ref_prior_theta}
\pi_i(\theta_i | \bs{\theta}_{-i})  \propto 
\sqrt{ \Tr \left[ 
\left\{ \bs{W} \trans \derivee  \bs{W}
 \left( \bs{W} \trans \corr \bs{W} \right)^{-1} \right\}^2 \right] - \frac{1}{ n-p } \left[ \Tr \left\{ \bs{W} \trans \derivee \bs{W} \left( \bs{W} \trans \corr \bs{W} \right)^{-1} \right\}\right]^2 }.
\end{equation}

Now consider the sequence of conditional posterior distributions $(\pi_i(\theta_i | \bs{y}, \bs{\theta}_{-i}))_{i \in [\![1,r]\!]}$. These conditional distributions are incompatible in the sense that there exists no joint probability distribution $\pi(\bs{\theta}|\bs{y})$ which agrees with all of them. We may however define the Gibbs reference posterior as a compromise between the conditionals in this sequence. In \citet{Mur18} we provided theoretical foundation for what such a compromise could be. In the end, we showed it to be the stationary probability distribution of a Markovian kernel $P_{\bs{y}}: (0,+\infty)^r \times \mathcal{B}\left((0,+\infty)^r \right)$, where $\mathcal{B}\left((0,+\infty)^r \right)$ denotes the Borel algebra on $\left((0,+\infty)^r \right)$. $P_{\bs{y}}$ is defined by the following expression, where $\bs{\theta}^{(0)} \in (0,1)^r$ and $\delta_t$ denotes the shifted Dirac measure $\delta(\cdot - t)$:

\begin{align}
P_{\bs{y}}(\bs{\theta}^{(0)}, d \bs{\theta}) &= \frac{1}{r} \sum_{i=1}^r \pi_i(\theta_i | \bs{y}, \bs{\theta}_{-i}^{(0)}) d \theta_i \;  \delta_{\bs{\theta}_{-i}^{(0)}} (d \bs{\theta}_{-i})  .
\end{align}

The goal of this section is to provide sufficient conditions for the existence (and thus, propriety) of this stationary probability distribution $\pi_G(\bs{\theta} | \bs{y})$ and to show that the Markov Chain Monte-Carlo (MCMC) algorithm based on the Markovian kernel $P_{\bs{y}}$ converges to it, that is, $P_{\bs{y}}$ is uniformly ergodic. This means that denoting $P^n_{\bs{y}}$ the Markov kernel produced by $n$ successive applications of $P_{\bs{y}}$ and $\| \cdot \|_{TV}$ the total variation norm,

\begin{equation} \label{Eq:uniforme_ergodicite}
\lim_{n \to \infty} \sup_{\bs{\theta}^{(0)} \in (0,+\infty)^r } \| P_{\bs{y}}^n (\bs{\theta}^{(0)},\cdot) - \pi_G(\cdot | \bs{y}) \|_{TV} = 0.
\end{equation} \medskip

In the following results, when we write that ``$P_{\bs{y}}$ is uniformly ergodic'', we mean that Equation (\ref{Eq:uniforme_ergodicite}) holds. \medskip

\subsection{Existence}

The results in this subsection deal with the following setting:

\begin{itemize}
\item The spatial domain is the unit cube $(0,1)^r$ ($r>0$).
\item The mean function space $\mathcal{F}_p$ has dimension $p \geqslant 0$.
\item The Universal Kriging model uses a Matérn anisotropic geometric or tensorized correlation kernel with smoothness parameter $\nu>0$.
\item Design sets contain $n>0$ points, so we identify $(0,1)^{rn}$ with the set of all design sets in the spatial domain $(0,1)^r$. Let $Q(r,n)$ be the Lebesgue measure on $(0,1)^{rn}$.
\end{itemize}

In the following, we change parametrization for the sake of convenience : define $\bs{\mu}$ such that $\forall i \in [\![1,r]\!]$, $\mu_i = 1 / \theta_i$. The conditionals are invariant to such a change, and therefore 
both the Markovian kernel $P_{\bs{y}}$  and, if it exists, its stationary probability remain the same. Abusing notations, the likelihood $L^1(\bs{y}|\bs{\theta})$ is denoted $L(\bs{y}|\bs{\mu})$ when expressed in the $\bs{\mu}$-parametrization.

\newcommand{\norme}{ \| \bs{\mu} \| }
\newcommand{\manqueI}{\bs{\mu}_{-i}}
\newcommand{\normeI}{ \| \bs{\mu}_{-i} \| }
\newcommand{\normeMu}{ \| \bs{\mu} \| }
\newcommand{\coIrr}{ \bs{ \Sigma }_{ \bs{ \mu } } }
\newcommand{\coI}{ \bs{ \Sigma }_{ \mu_i } }
\newcommand{\co}{ \bs{ \Sigma }_{ \mu } }
\newcommand{\coIrrI}{\bs{\Sigma}_{\bs{\mu}_{-i}}}
\newcommand{\deIrivee}{ \frac{\partial}{\partial \mu_i } \coIrr }
\newcommand{\deI}{ \frac{d}{d \mu_i } \coI }

Define the functions $f_i$ by

\begin{equation} \label{Eq:Gibbs_ref_prior_mu}
f_i(\mu_i | \manqueI ) :=
\sqrt{\Tr \left[ \left( \bs{W} \trans \deIrivee \bs{W} \left( \bs{W} \trans \coIrr \bs{W} \right)^{-1} \right)^2 \right] 
-
\frac{1}{n-p} \Tr \left[\bs{W} \trans \deIrivee \bs{W} \left( \bs{W} \trans \coIrr \bs{W} \right)^{-1} \right]^2 }.
\end{equation}

Then, following Equation (\ref{Eq:Gibbs_ref_prior_theta}), the conditional density $\pi_i$ is in the $\bs{\mu}$-parametrization given by:

\begin{equation}
\pi_i(\mu_i | \manqueI) \propto f_i(\mu_i | \manqueI).
\end{equation}

We need to make some assumptions which are detailed below.

\begin{hyp} \label{Hyp:moyenne_non_zero}
Any vector in the subspace of $\R^n$ spanned by $\bs{H}$ is either null or has strictly more than $2r$ non-null elements when expressed in the canonical base.

\end{hyp}

\begin{rmq}
It is not apparent, but the purpose of Assumption \ref{Hyp:moyenne_non_zero} is to control the behavior of the $f_i(\mu_i | \manqueI)$ ($i \in [\![1,r]\!]$) when $\norme \to \infty$. See the proofs in Appendix \ref{App:Gibbs_posterior} for details.
\end{rmq}

This assumption is not very restrictive, as the two following results show.

\begin{prop}
In Ordinary Kriging -- that is with $p=1$ and $\mathcal{F}_p$ being the space of constant functions -- if $n>2r$, Assumption \ref{Hyp:moyenne_non_zero} is automatically verified.
\end{prop}
\begin{proof}
In this setting, $\bs{H}$ is a non-null constant $n \times 1$ matrix, so Assumption \ref{Hyp:moyenne_non_zero} is trivially verified.
\end{proof}

\begin{prop} \label{Prop:krigeage_affine_hyp_moyenne_non_zero}
Assume that the design set is such that any subset with cardinal $r+1$ forms a simplex.
Then in Universal Kriging, if the mean function space $\mathcal{F}_p$ is included within the vector space of polynomials of degree 0 and 1, and if $n>3r$, Assumption  \ref{Hyp:moyenne_non_zero} is automatically verified.
\end{prop}
\begin{proof}
Let $\bs{y}^*$ belong to the subspace of $\R^n$ spanned by $\bs{H}$. Assume that it has $2r$ or fewer non-null elements when expressed in the canonical base. Conversely, it has at least $n-2r$ null elements. If $n>3r$, then this means that there exists a function $f^* \in \mathcal{F}_p$ (the one represented by $\bs{y}^*$) which admits at least $r+1$ zeros on the design set. However, given the premise of Proposition \ref{Prop:krigeage_affine_hyp_moyenne_non_zero}, these $r+1$ points form a simplex, so they span an affine space of dimension $r$. As $f^*$ is a polynomial with $r$ unknowns of degree 0 or 1, this implies that $f^*=0$.
\end{proof}

\begin{rmq}
$Q(r,n)$-almost all design sets fit 
the premise of Proposition \ref{Prop:krigeage_affine_hyp_moyenne_non_zero}.
\end{rmq}

In some cases, Assumption \ref{Hyp:moyenne_non_zero} is sufficient for our purposes. Define $\bs{1}$ as the vector of $\R^n$ with all components in the canonical basis equal to 1.

\begin{prop} \label{Prop:nu<1}
In the setting described above, if $0<\nu<1$ and $n>p+1$, then
for $Q(r,n)$-almost all design sets, if $\bs{1}$ does not belong to the vector space spanned by $\bs{H}$, then
 Assumption \ref{Hyp:moyenne_non_zero} implies that 
there exists a hyperplane $\mathcal{H}$ of $\R^n$ such that $\forall \bs{y} \in \R^n \setminus \mathcal{H}$, 
$P_{\bs{y}}$ is uniformly ergodic.
\end{prop}

The proof of this Proposition can be found in Appendix \ref{App:Gibbs_posterior}.

Naturally, the above result is somewhat unsatisfactory since most users will want to include non-null constant functions in $\mathcal{F}_p$.

Consider now the following assumption.

\begin{hyp} \label{Hyp:vraisemblance_discriminante}
There exists $\epsilon_{\bs{y}}>0$ such that $L(\bs{y} | \bs{\mu}) = O(\norme^{\epsilon_{\bs{y}}})$ when $\norme \to 0$.
\end{hyp}

\begin{rmq}
Assumption \ref{Hyp:vraisemblance_discriminante} essentially means that the model should find perfect correlation unlikely.
\end{rmq}

The following theorem, which is proved in Appendix \ref{App:Gibbs_posterior}, is our essential tool for dealing with the case where non-null constant functions are included in $\mathcal{F}_p$.

\begin{thm} \label{Thm:compromis_gibbs_existe}
In the setting described above, if $\nu>1$ and $n>p+r+2$, then for $Q(r,n)$-almost all design sets,
 Assumptions \ref{Hyp:moyenne_non_zero} and \ref{Hyp:vraisemblance_discriminante} imply that $P_{\bs{y}}$ is uniformly ergodic.
\end{thm}

The next two results, which are proved in Appendix \ref{App:Gibbs_posterior}, concern particular settings where Assumptions \ref{Hyp:moyenne_non_zero} and \ref{Hyp:vraisemblance_discriminante} are both verified and therefore Theorem \ref{Thm:compromis_gibbs_existe} yields the uniform ergodicity of $P_{\bs{y}}$.

\begin{prop} \label{Prop:compromis_gibbs_existe_kriOrdin}
Consider the particular case of the above described setting where $p=1$ and $\mathcal{F}_p$ is the space of all constant functions (Ordinary Kriging), and assume that one of the following conditions is satisfied:
\begin{enumerate}
\item  $1<\nu<2$ and $n>r+3$;
\item  $2<\nu<3$ and $n>(r+1)(r/2+2)$.
\end{enumerate}
Then, for $Q(r,n)$-almost all design sets, there exists a hyperplane $\mathcal{H}$ of $\R^n$ such that $\forall \bs{y} \in \R^n \setminus \mathcal{H}$, 
$P_{\bs{y}}$ is uniformly ergodic.

\end{prop}

\begin{prop} \label{Prop:compromis_gibbs_existe_kriD1}
Consider the particular case of the above described setting where $\mathcal{F}_p$ is included within the space of all polynomials of degree 0 and 1 (so $p\leqslant r+1$) and assume that the following condition is satisfied:
\begin{itemize}
\item $2<\nu<3$ and $n>r(r+1)/2 + 2r + 3$.
\end{itemize}
Then, for $Q(r,n)$-almost all design sets, there exists a hyperplane $\mathcal{H}$ of $\R^n$ such that $\forall \bs{y} \in \R^n \setminus \mathcal{H}$, 
$P_{\bs{y}}$ is uniformly ergodic.
\end{prop}

\begin{rmq}
In Propositions \ref{Prop:nu<1}, \ref{Prop:compromis_gibbs_existe_kriOrdin} and \ref{Prop:compromis_gibbs_existe_kriD1}, the condition that the observation $\bs{y}$ should not belong to a given negligible (for the Lebesgue measure) subset of $\R^n$ is fairly natural: for the Kriging model to be adequate, 
$\bs{y}$ must not look like a realization of a degenerate Gaussian vector.
Theorem \ref{Thm:compromis_gibbs_existe} does not really dispense with it, as it is implied by Assumption \ref{Hyp:vraisemblance_discriminante}.
\end{rmq}

To sum up the results of this section, to ensure that the Gibbs reference posterior exists and can be accessed through Gibbs sampling, one should check that one of the following assertions is true:

\begin{itemize}
\item $\nu>1$, $n>r+p+2$ and both Assumptions \ref{Hyp:moyenne_non_zero} and \ref{Hyp:vraisemblance_discriminante} are verified;
\item $\mathcal{F}_p$ contains only constant functions and $1 < \nu < 2$ and $n>r+3$;

\item $\mathcal{F}_p$ contains only polynomials of degree 0 and 1, $2 < \nu < 3$ and $n>r(r+1)/2 + 2r + 3$;
\item $0<\nu<1$, $n>p+1$, no non-null constant function belongs to $\mathcal{F}_p$ and Assumption \ref{Hyp:moyenne_non_zero} is verified.
\end{itemize}

\section{Comparison of the predictive performance of the full-Bayesian approach versus MLE and MAP plug-in approaches}
\label{Sec:comp_predictions}

In this section, we evaluate the predictive performance resulting from the Gibbs reference posterior distribution $\pi_G(\bs{\theta} | \bs{y})$ in the context of a well-specified model, and then when emulating some deterministic real functions. We contrast the full-Bayesian approach, in which the Full Gibbs reference Posterior Distribution (FPD) is used, with two plug-in approaches: one where the Maximum Likelihood Estimator (MLE) and the other where the Maximum A Posteriori (MAP) estimator is assumed to be the true value of $\bs{\theta}$. All approaches make use of the reference posterior $\pi(\bs{\beta},\sigma^2 | \bs{y}, \bs{\theta})$. \medskip

We use the following terminology. We call Simple Kriging the Kriging model where the mean function is assumed to be known, whether this assumption is correct or known.
We call Ordinary Kriging any Universal Kriging model where the mean function space is the space of constant functions.
We call Affine Kriging any Universal Kriging model where the mean function space is the space of affine functions.

\subsection{Well-specified model}

We first consider well-specified models, specifically Kriging models with unknown parameters $(\bs{\beta},\sigma^2, \bs{\theta})$ emulating actual Gaussian processes with variance $\sigma^2=1$ and Matérn anisotropic geometric autocorrelation kernel with smoothness $\nu=5/2$. Moreover, the true mean function of the Gaussian process belongs to the assumed mean function space $\mathcal{F}_p$. \medskip

The spatial domain is the unit cube $(0,1)^r$ and the considered design sets all contain $n$ points independently chosen according to the Lebesgue measure on the domain $(0,1)^r$. \medskip

The following tables give the average coverage and average mean length of prediction intervals. To define these notions, we introduce the following notations:

\begin{itemize}
\item $\bs{Y}$ is the Gaussian process, and $\bs{Y}(\bs{x})$ is the vector of the values taken by said process at the points in the design set $\bs{x}$;
\item $T$ is a random variable which follows the Uniform distribution on the unit cube $(0,1)^r$. It represents the ``test'' point;
\item $\bs{X}$ is the random design set following the Uniform distribution on $\left((0,1)^{r}\right)^{n}$.
\item $\bs{Y}$, $T$ and $\bs{X}$ are mutually independent;
\item $f$ is a function defined on $\left((0,1)^{r}\right)^{n} \times \R^{n} \times (0,1)^r$ which associates to $(\bs{x},\bs{y},t)$ the prediction interval at $t$ of the Gaussian process, based on the knowledge of its value $\bs{y}$ on the design set $\bs{x}$.
\end{itemize}

\begin{defn}
The average coverage is the probability (with respect to the distributions of $\bs{X}$, $\bs{Y}$ and $T$) that $\bs{Y}(T) \in f(\bs{X},\bs{Y}(\bs{X}),T)$.
\end{defn}

\begin{defn}
The average mean length is the expectation (with respect to the distributions of $\bs{X}$, $\bs{Y}$ and $T$) of the length of $ f(\bs{X},\bs{Y}(\bs{X}),T)$.
\end{defn}

The average coverage $\mathbb{P} [\bs{Y}(T) \in f(\bs{X},\bs{Y}(\bs{X}),T)]$ is numerically computed as

\[ 
\mathbb{P} [\bs{Y}(T) \in f(\bs{X},\bs{Y}(\bs{X}),T)] = \E[\mathbb{P}[\bs{Y}(T) \in f(\bs{X},\bs{Y}(\bs{X}),T) | \bs{X}, \bs{Y}(\bs{X})] ]
\]

 over 500 random design sets and for each design set 1000 random test points. The average mean length is computed in a similar fashion. \medskip

In this subsection and the following one, we take $n=30$ and $r=3$. \medskip

In the first set of simulations, we use a well-specified Ordinary Kriging model, with the unknown mean 5. As $r=3$ and $n=30$, Proposition \ref{Prop:compromis_gibbs_existe_kriOrdin} is applicable. \medskip

The results given in Table \ref{Tab:IP_taux_couverture_krigeage_ordinaire} show that using the full posterior distribution (FPD) to derive the predictive distribution is the best possible choice from a frequentist point of view as the nominal value is nearly matched by the average coverage. Predictive Intervals derived from the MAP estimator do not perform as well, and Predictive Intervals derived from the MLE perform even worse. \medskip

\begin{table}[!ht]
\begin{center}
\begin{tabular}{|c|c|c|c|c|}
\hline 
 \multicolumn{5}{|c|}{\textbf{Average Coverage}} \\
\hline
\textbf{Corr. lengths} & \textbf{True} & \textbf{MLE} & \textbf{MAP} & \textbf{FPD} \\ 
\hline 
0.4 – 0.8 – 0.2 & 0.95 & 0.88 & 0.91 & 0.95 \\

0.5 – 0.5 – 0.5 & 0.95 & 0.88 & 0.90 & 0.94 \\

0.7 – 1.3 – 0.4 & 0.95 & 0.90 & 0.92 & 0.95 \\

0.8 – 0.3 – 0.6 & 0.95 & 0.89 & 0.91 & 0.94 \\

0.8 – 1.0 – 0.9 & 0.95 & 0.90 & 0.92 & 0.94 \\
\hline
\end{tabular} 
\caption{For a Gaussian Process with constant mean function equal to 5, variance parameter 1 and smoothness parameter 5/2, average coverage of 95\% Prediction Intervals produced by an Ordinary Kriging model. ``True'' stands for the Simple Kriging prediction based on the knowledge of the true mean parameter, variance parameter and vector of correlation lengths.}
\label{Tab:IP_taux_couverture_krigeage_ordinaire}
\end{center}
\end{table}

The results given in Table \ref{Tab:IP_longueur_moyenne_krigeage_ordinaire} show that Predictive Intervals arising from the full Gibbs reference posterior distribution (FPD) are on average somewhat larger than those resulting from knowledge of the true parameters, while intervals arising from both types of parameter estimation (MLE and MAP) are too short. \medskip

\begin{table}[!ht]
\begin{center}
\begin{tabular}{|c|c|c|c|c|}
\hline 
 \multicolumn{5}{|c|}{\textbf{Average Mean Length}} \\
\hline
\textbf{Corr. lengths} & \textbf{True} & \textbf{MLE} & \textbf{MAP} & \textbf{FPD} \\ 
\hline 
0.4 – 0.8 – 0.2 & 2.23 & 2.06 & 2.14 & 2.58 \\

0.5 – 0.5 – 0.5 & 1.69 & 1.55 & 1.59 & 1.83  \\

0.7 – 1.3 – 0.4 & 1.09 & 1.02 & 1.07 & 1.20 \\

0.8 – 0.3 – 0.6 & 1.63 & 1.51 & 1.57 & 1.81 \\

0.8 – 1.0 – 0.9 & 0.71 & 0.66 & 0.69 & 0.76 \\
\hline
\end{tabular} 
\caption{For a Gaussian Process with constant mean function equal to 5, variance parameter 1 and smoothness parameter 5/2, average mean length of 95\% Prediction Intervals produced by an Ordinary Kriging model. ``True'' stands for the Simple Kriging prediction based on the knowledge of the true mean parameter, variance parameter and vector of correlation lengths.}
\label{Tab:IP_longueur_moyenne_krigeage_ordinaire}
\end{center}
\end{table}

Consider now Universal Kriging models where the true mean function is the polynomial $(x_1,x_2,x_3) \mapsto 5 + 4x_1 + 3 x_2 + 2 x_3$, and the model (correctly) assumes that it belongs to the 4-dimensional space ($p=4$) spanned by the functions mapping $(x_1,x_2,x_3)$ to $1$, $x_1$, $x_2$ and $x_3$ respectively. For such Affine Kriging models, Proposition \ref{Prop:compromis_gibbs_existe_kriD1} is applicable. \medskip

As shown in Table \ref{Tab:IP_taux_couverture_krigeage_affine}, Predictive Intervals resulting from both plug-in approaches (MLE, MAP) and from the full posterior distribution perform a little worse than in the Ordinary Kriging setting, but their relative performances stay the same. \medskip

\begin{table}
\begin{center}
\begin{tabular}{|c|c|c|c|c|}
\hline 
 \multicolumn{5}{|c|}{\textbf{Average Coverage}} \\
\hline 
\textbf{Corr. lengths} & \textbf{True} & \textbf{MLE} & \textbf{MAP} & \textbf{FPD} \\ 
\hline 
0.4 – 0.8 – 0.2 & 0.95 & 0.87 & 0.90 & 0.94 \\

0.5 – 0.5 – 0.5 & 0.95 & 0.87 & 0.89 & 0.92 \\

0.7 – 1.3 – 0.4 & 0.95 & 0.89 & 0.92 & 0.94 \\

0.8 – 0.3 – 0.6 & 0.95 & 0.87 & 0.90 & 0.93 \\

0.8 – 1.0 – 0.9 & 0.95 & 0.89 & 0.92 & 0.93 \\
\hline
\end{tabular} 
\caption{For a Gaussian Process with mean function $(x_1,x_2,x_3) \mapsto 5 + 4x_1 + 3x_2 + 2x_3$, variance parameter 1 and smoothness parameter 5/2, average coverage of 95\% Prediction Intervals produced by an Affine Kriging model. ``True'' stands for the Simple Kriging prediction based on the knowledge of the true mean function, variance parameter and vector of correlation lengths.}
\label{Tab:IP_taux_couverture_krigeage_affine}
\end{center}
\end{table}

Table \ref{Tab:IP_longueur_moyenne_krigeage_affine}  shows that the average mean lengths of Predictive Intervals are not very different in Affine Kriging than in Ordinary Kriging when it comes to the FPD. However, they are larger in Affine Kriging than in Ordinary Kriging when it comes to the MLE and the MAP.
Interestingly, Predictive Intervals resulting from the MAP have about the same size as Predictive Intervals derived when all parameters are known. Those derived using the MLE are shorter, and those derived from the FPD are larger.

\begin{table}[!ht]
\begin{center}
\begin{tabular}{|c|c|c|c|c|}
\hline 
 \multicolumn{5}{|c|}{\textbf{Average Mean Length}} \\
\hline 
\textbf{Corr. lengths} & \textbf{True} & \textbf{MLE} & \textbf{MAP} & \textbf{FPD} \\ 
\hline 
0.4 – 0.8 – 0.2 & 2.23 & 2.14 & 2.23 & 2.59 \\

0.5 – 0.5 – 0.5 & 1.69 & 1.57 & 1.66 & 1.83 \\

0.7 – 1.3 – 0.4 & 1.09 & 1.04 & 1.10 & 1.20 \\

0.8 – 0.3 – 0.6 & 1.63 & 1.54 & 1.61 & 1.80 \\

0.8 – 1.0 – 0.9 & 0.71 & 0.67 & 0.71 & 0.75 \\
\hline
\end{tabular} 
\caption{For a Gaussian Process with mean function $(x_1,x_2,x_3) \mapsto 5 + 4x_1 + 3x_2 + 2x_3$, variance parameter 1 and smoothness parameter 5/2, average mean length of 95\% Prediction Intervals produced by an Affine Kriging model. ``True'' stands for the Simple Kriging prediction based on the knowledge of the true mean function, variance parameter and vector of correlation lengths.}
\label{Tab:IP_longueur_moyenne_krigeage_affine}
\end{center}
\end{table}

For reference, we give the tables obtained in the Simple Kriging case, that is the case where the Gaussian Process is known to have null mean function. Table \ref{Tab:IP_taux_couverture} gives the average coverages and Table \ref{Tab:IP_longueur_moyenne} the average mean lengths.

\begin{table}[!ht]
\begin{center}
\begin{tabular}{|c|c|c|c|c|}
\hline 
 \multicolumn{5}{|c|}{\textbf{Average Coverage}} \\
\hline
\textbf{Corr. lengths} & \textbf{True} & \textbf{MLE} & \textbf{MAP} & \textbf{FPD} \\ 
\hline 
0.4 – 0.8 – 0.2 & 0.95 & 0.88 & 0.91 & 0.95 \\

0.5 – 0.5 – 0.5 & 0.95 & 0.89 & 0.90 & 0.94 \\
 
0.7 – 1.3 – 0.4 & 0.95 & 0.90 & 0.92 & 0.95 \\

0.8 – 0.3 – 0.6 & 0.95 & 0.89 & 0.91 & 0.95 \\

0.8 – 1.0 – 0.9 & 0.95 & 0.90 & 0.92 & 0.94 \\
\hline
\end{tabular} 
\caption{For a Gaussian Process with null mean function, variance parameter 1 and smoothness parameter 5/2, average coverage of 95\% Prediction Intervals produced by a Simple Kriging model. ``True'' stands for the prediction based on the knowledge of the true variance parameter and the true vector of correlation lengths.}
\label{Tab:IP_taux_couverture}
\end{center}
\end{table}

\begin{table}[!ht]
\begin{center}
\begin{tabular}{|c|c|c|c|c|}
\hline 
 \multicolumn{5}{|c|}{\textbf{Average Mean Length}} \\
\hline
\textbf{Corr. lengths} & \textbf{True} & \textbf{MLE} & \textbf{MAP} & \textbf{FPD} \\ 
\hline 
0.4 – 0.8 – 0.2 & 2.23 & 2.05 & 2.13 & 2.59 \\

0.5 – 0.5 – 0.5 & 1.69 & 1.55 & 1.58 & 1.84 \\

0.7 – 1.3 – 0.4 & 1.09 & 1.02 & 1.07 & 1.21 \\

0.8 – 0.3 – 0.6 & 1.63 & 1.51 & 1.56 & 1.82 \\

0.8 – 1.0 – 0.9 & 0.71 & 0.66 & 0.69 & 0.76 \\
\hline
\end{tabular} 
\caption{For a Gaussian Process with null mean function, variance parameter 1 and smoothness parameter 5/2, average mean length of 95\% Prediction Intervals produced by a Simple Kriging model. ``True'' stands for the prediction based on the knowledge of the true variance parameter and the true vector of correlation lengths.}
\label{Tab:IP_longueur_moyenne}
\end{center}
\end{table}

The performance of Ordinary Kriging when the mean function is constant is nearly the same as that of Simple Kriging when the mean function is known. 

The performance of Affine Kriging when the mean function is affine, however, is noticeably poorer than the performance of Simple Kriging when the mean function is known:
its average coverage is lower. This is not too surprising, since the prediction problem is more difficult.

\subsection{Misspecified models}

In this subsection, we deal with the performance of Kriging in cases where the Gaussian Process does not fit all assumptions. \medskip

First, we evaluate the performance of Universal Kriging in a context where the true mean function does not belong to the assumed mean function space $\mathcal{F}_p$. Precisely, we consider a Gaussian process with mean function $(x_1,x_2,x_3) \mapsto 5 + 4x_1 + 3x_2 + 2 x_3$ and evaluate the performance of Simple Kriging (assuming the mean function to be null) with respect to that of Affine Kriging, which is the correct model in this situation.

Tables \ref{Tab:IP_taux_couverture_krigeage_simple_moyenne_affine} and \ref{Tab:IP_longueur_moyenne_krigeage_simple_moyenne_affine} show that Simple Kriging performs significantly worse than Affine Kriging when the mean function is $(x_1,x_2,x_3) \mapsto 5 + 4x_1 + 3x_2 + 2x_3$, both in terms of average coverage and average mean length of Predictive Intervals. 
Relative performances of MLE, MAP and FPD once again stay the same, though.

\begin{table}
\begin{center}
\begin{tabular}{|c|c|c|c|c|}
\hline 
 \multicolumn{5}{|c|}{\textbf{Average Coverage}} \\
\hline 
\textbf{Corr. lengths} & \textbf{True} & \textbf{MLE} & \textbf{MAP} & \textbf{FPD} \\ 
\hline 
0.4 – 0.8 – 0.2 & 0.95 & 0.77 & 0.81 & 0.88 \\

0.5 – 0.5 – 0.5 & 0.95 & 0.80 & 0.82 & 0.89 \\

0.7 – 1.3 – 0.4 & 0.95 & 0.82 & 0.86 & 0.91 \\

0.8 – 0.3 – 0.6 & 0.95 & 0.79 & 0.83 & 0.89 \\

0.8 – 1.0 – 0.9 & 0.95 & 0.82 & 0.86 & 0.91 \\
\hline
\end{tabular} 
\caption{For a Gaussian Process with mean function $(x_1,x_2,x_3) \mapsto 5 + 4x_1 + 3 x_2 + 2 x_3$, variance parameter 1 and smoothness parameter 5/2, average coverage of 95\% Prediction Intervals resulting from Simple Kriging (assuming the mean function is null for MLE/MAP/FPD and knowing it is $(x_1,x_2,x_3) \mapsto 5 + 4x_1 + 3 x_2 + 2 x_3$ for ``True''). ``True'' stands for the prediction based on the knowledge of the true mean function, variance parameter and vector of correlation lengths.}
\label{Tab:IP_taux_couverture_krigeage_simple_moyenne_affine}
\end{center}
\end{table}

\begin{table}[!ht]
\begin{center}
\begin{tabular}{|c|c|c|c|c|}
\hline 
 \multicolumn{5}{|c|}{\textbf{Average Mean Length}} \\
\hline 
\textbf{Corr. lengths} & \textbf{True} & \textbf{MLE} & \textbf{MAP} & \textbf{FPD} \\ 
\hline 
0.4 – 0.8 – 0.2 & 2.23 & 2.23 & 2.36 & 2.78 \\

0.5 – 0.5 – 0.5 & 1.69 & 1.61 & 1.66 & 1.92  \\

0.7 – 1.3 – 0.4 & 1.09 & 1.03 & 1.13 & 1.28 \\

0.8 – 0.3 – 0.6 & 1.63 & 1.54 & 1.63 & 1.87 \\

0.8 – 1.0 – 0.9 & 0.71 & 0.64 & 0.68 & 0.77  \\
\hline
\end{tabular} 
\caption{For a Gaussian Process with mean function $(x_1,x_2,x_3) \mapsto 5 + 4x_1 + 3 x_2 + 2 x_3$, variance parameter 1 and smoothness parameter 5/2, average mean length of 95\% Prediction Intervals resulting from Simple Kriging (assuming the mean function is null for MLE/MAP/FPD and knowing it is $(x_1,x_2,x_3) \mapsto 5 + 4x_1 + 3 x_2 + 2 x_3$ for ``True''). ``True'' stands for the prediction based on the knowledge of the true mean function, variance parameter and vector of correlation lengths.}
\label{Tab:IP_longueur_moyenne_krigeage_simple_moyenne_affine}
\end{center}
\end{table}

This observation may lead us to investigate how Simple Kriging behaves with respect to Affine Kriging when the Gaussian Process is smoother than expected. Table \ref{Tab:IP_GASP_smoothness_inf} gives the average coverage and average mean length of Prediction Intervals resulting from the same procedure as before -- that is, the correlation kernel is assumed to be Matérn with smoothness 5/2 -- but the Gaussian Process actually has a Squared Exponential correlation kernel (with correlation lengths 0.4, 0.8 and 0.2). 
These results can be compared with those from Table \ref{Tab:IP_GASP_smoothness_5_demis}, which gives the results obtained when both the actual and the assumed correlation kernel are Matérn with smoothness 5/2 (and the true correlation lengths are also 0.4, 0.8 and 0.2).
It is apparent that performance is better when the actual kernel is Squared Exponential, both in terms of average coverage and average mean length. Recalling that this kernel can be seen as the limit of the Matérn kernel when the smoothness parameter goes to infinity, we conclude that a smoother process leads to an increase in performance for Simple, Ordinary and Affine Kriging.
For Affine Kriging, the smoother process makes Prediction Intervals on average shorter, while the average coverage remains about the same.
For Simple Kriging and to a lesser degree Ordinary Kriging, the smoother process makes Prediction Intervals on average shorter, while also increasing average coverage.

\begin{table}[!ht]
\begin{center}
\begin{tabular}{|c||c|c|c||c|c|c|}
\hline
\textsc{Squared Exponential Correlation Kernel}& \multicolumn{3}{c||}{\textbf{Average coverage}} & \multicolumn{3}{c|}{\textbf{Average mean length}} \\
\hline
\hline
\textbf{Kriging model} & \textbf{MLE} & \textbf{MAP} & \textbf{FPD} & \textbf{MLE} & \textbf{MAP} & \textbf{FPD}\\ 
\hline 
Simple Kriging (mean function assumed null) & 0.83 & 0.86 & 0.92 & 1.63 & 1.76 & 2.02 \\

Ordinary Kriging & 0.88 & 0.90 & 0.93 & 1.70 & 1.79 & 2.01 \\

Affine Kriging & 0.89 & 0.91 & 0.93 & 1.63 & 1.70 & 1.88 \\
\hline
\end{tabular} 
\caption{For a Gaussian Process with mean function $(x_1,x_2,x_3) \mapsto 5 + 4x_1 + 3 x_2 + 2 x_3$, variance parameter 1, and squared exponential correlation kernel with correlation lengths 0.4 - 0.8 - 0.2, average coverage and average mean length of 95\% Prediction Intervals resulting from different types of Kriging (assuming the smoothness parameter to be 5/2).}
\label{Tab:IP_GASP_smoothness_inf}
\end{center}
\end{table}

\begin{table}[!ht]
\begin{center}
\begin{tabular}{|c||c|c|c||c|c|c|}
\hline
\textsc{Matérn kernel with smoothness 5/2}& \multicolumn{3}{c||}{\textbf{Average coverage}} & \multicolumn{3}{c|}{\textbf{Average mean length}} \\
\hline
\hline
\textbf{Kriging model} & \textbf{MLE} & \textbf{MAP} & \textbf{FPD} & \textbf{MLE} & \textbf{MAP} & \textbf{FPD}\\ 
\hline 
Simple Kriging (mean function assumed null) & 0.77 & 0.81 & 0.88 & 2.23 & 2.36 & 2.77 \\

Ordinary Kriging & 0.84 & 0.86 & 0.91 & 2.30 & 2.37 & 2.71 \\

Affine Kriging & 0.87 & 0.90 & 0.94 & 2.14 & 2.23 & 2.59 \\
\hline
\end{tabular} 
\caption{For a Gaussian Process with mean function $(x_1,x_2,x_3) \mapsto 5 + 4x_1 + 3 x_2 + 2 x_3$, variance parameter 1 and Matérn kernel with correlation lengths 0.4 - 0.8 - 0.2 and smoothness parameter 5/2, average coverage and average mean length of 95\% Prediction Intervals resulting from different types of Kriging.}
\label{Tab:IP_GASP_smoothness_5_demis}
\end{center}
\end{table}

All else being equal, smoother processes result in a better quality of prediction for Simple, Ordinary and Affine Kriging, because the observed values of the process yield more information about the value of the process in the neighborhoods of the observation points. This even makes up to some degree for the misspecification of the mean function, so the improvement is greater in the case of Simple Kriging. \medskip

\subsection{Emulating deterministic functions}

In this subsection, we test the ability of the model to predict deterministic functions, namely the 7-dimensional Ackley and Rastrigin functions. The Ackley and the Rastrigin functions have the following expressions:

\begin{align}
A(\bs{x}) &= 20 + \exp(1) - 20 \exp \left( - 0.2 \sqrt{ \frac{1}{7} \sum_{i=1}^{7} x_i^2} \right) - \exp \left( \frac{1}{7} \sum_{i=1}^{7} \cos(2 \pi x_i) \right) ; \\
R(\bs{x}) &= 70 + \sum_{i=1}^{7} \left( x_i^2 - 10 \cos(2 \pi x_i) \right).
\end{align}

Naturally, the notions of average coverage and average mean length for Prediction intervals make no sense in this setting, since we can no longer average our results over the distribution of a Gaussian process. Denoting $d$ the deterministic function, and using previous notations, we may define:

\begin{defn}
The coverage is the probability (with respect to the distribution of $\bs{X}$ and $T$) that $d(T) \in f(\bs{X},d(\bs{X}),T)$.
\end{defn}

\begin{defn}
The mean length is the expectation (with respect to the distribution of $\bs{X}$ and $T$) of the length of  $f(\bs{X},d(\bs{X}),T)$.
\end{defn}

The coverage $\mathbb{P} [\bs{Y}(T) \in f(\bs{X},d(\bs{X}),T)]$ is numerically computed as

$$\mathbb{P} [\bs{Y}(T) \in f(\bs{X},\bs{Y}(\bs{X}),T)] = \E[\mathbb{P}[\bs{Y}(T) \in f(\bs{X},d(\bs{X}),T) | \bs{X}] ]$$

 over 500 design sets and for each design set 1000 test points. The mean length is computed in a similar fashion. \medskip

When emulating the Ackley or the Rastrigin function, we take $r=7$ and $n=100$. \medskip

We must stress that there is no reason that the coverage of 95\% Prediction Intervals, whether produced by MLE or MAP plug-in methods or by the full Gibbs reference posterior distribution should be 95\%, but depending on whether or not Kriging can be considered a good surrogate model for the Ackley or Rastrigin function, the coverage of 95\% Prediction Intervals may be more or less close to the 95\% target figure. \medskip

First, we consider an Ordinary Kriging model with anisotropic geometric Matérn kernel of smoothness $\nu = 5/2$. Because $r=7$ and $n=100$, Proposition  \ref{Prop:compromis_gibbs_existe_kriOrdin} is applicable. \medskip

When emulating the Ackley function (cf. Table \ref{Tab:IP_Ackley}), regardless of the Kriging method used, the full posterior distribution significantly improves the average coverage of Prediction Intervals when compared to the MLE or the MAP, with a comparatively small trade-off regarding the mean length of these intervals. This result is consistent with results obtained with actual realizations of Gaussian processes.

When we emulate the Rastrigin function (cf. Table \ref{Tab:IP_Rastrigin}), coverages come closer to the average coverages given in Tables  \ref{Tab:IP_taux_couverture_krigeage_ordinaire}, \ref{Tab:IP_taux_couverture_krigeage_affine} and \ref{Tab:IP_taux_couverture}. But the more significant fact of the improvement of the coverage by the full posterior distribution is as true here as in the Ackley case. We may simply infer from this that the Rastrigin function can more plausibly be seen as a realization of a Gaussian Process than the Ackley function. \medskip

\begin{table}[!ht]
\begin{center}
\begin{tabular}{|c||c|c|c||c|c|c|}
\hline
\textsc{Emulated function: Ackley}& \multicolumn{3}{c||}{\textbf{Coverage}} & \multicolumn{3}{c|}{\textbf{Mean length}} \\
\hline
\hline 
\textbf{Kriging model} & \textbf{MLE} & \textbf{MAP} & \textbf{FPD} & \textbf{MLE} & \textbf{MAP} & \textbf{FPD}\\ 
\hline 
Simple Kriging & 0.84 & 0.87 & 0.90 & 0.35 & 0.36 & 0.39 \\
\hline
Ordinary Kriging & 0.87 & 0.88 & 0.91 & 0.37 & 0.38 & 0.41 \\
\hline
Affine Kriging & 0.87 & 0.90 & 0.91 & 0.37 & 0.39 & 0.41 \\
\hline
\end{tabular} 
\caption{Coverage and mean length of 95\% Prediction Intervals when emulating the 7-dimensional Ackley function (Matérn anisotropic geometric correlation kernel with smoothness $\nu=5/2$).}
\label{Tab:IP_Ackley}
\end{center}
\end{table}

\begin{table}[!ht]
\begin{center}
\begin{tabular}{|c||c|c|c||c|c|c|}
\hline
\textsc{Emulated function: Rastrigin}& \multicolumn{3}{c||}{\textbf{Coverage}} & \multicolumn{3}{c|}{\textbf{Mean length}} \\
\hline
\hline
\textbf{Kriging model} & \textbf{MLE} & \textbf{MAP} & \textbf{FPD} & \textbf{MLE} & \textbf{MAP} & \textbf{FPD}\\ 
\hline 
Simple Kriging & 0.94 & 0.94 & 0.96 & 28.3 & 28.3 & 30.2 \\
\hline
Ordinary Kriging & 0.91 & 0.92 & 0.94 & 26.2 & 26.7 & 28.3 \\
\hline
Affine Kriging & 0.90 & 0.91 & 0.92 & 25.9 & 26.5 & 27.2 \\
\hline
\end{tabular} 
\caption{Coverage and mean length of 95\% Prediction Intervals when emulating the 7-dimensional Rastrigin function (Matérn anisotropic geometric correlation kernel with smoothness $\nu=5/2$).}
\label{Tab:IP_Rastrigin}
\end{center}
\end{table}

Let us now compare the performance of different Kriging models: Simple (mean function assumed null), Ordinary and Affine. When emulating the Ackley function, Ordinary and Affine Kriging models yield slightly higher Prediction Interval coverages than Simple Kriging, at the cost of slightly higher mean lengths. When emulating the Rastrigin function, we actually observe the reverse phenomenon.

From this study, we can not conclusively ascertain whether Universal Kriging, at least in the form of Ordinary or Affine Kriging, yields better results than Simple Kriging. All that can be said is that these Kriging methods are more or less conservative, but even this depends on the emulated function. 

In the following example (cf. Table \ref{Tab:IP_Rastrigin_lin}), we add the linear function $(x_1,x_2,x_3,x_4,x_5,x_6,x_7) \mapsto 100 \sum_{i=1}^7 x_i$ to the 7-dimensional Rastrigin function. We may expect this modification of the Rastrigin function to be more accurately emulated by Affine Kriging than by Simple Kriging.

\begin{table}[!ht]
\begin{center}
\begin{tabular}{|c||c|c|c||c|c|c|}
\hline
\textsc{Rastrigin + $100 \sum_{i=1}^7 x_i$}& \multicolumn{3}{c||}{\textbf{Coverage}} & \multicolumn{3}{c|}{\textbf{Mean length}} \\
\hline
\hline
\textbf{Kriging model} & \textbf{MLE} & \textbf{MAP} & \textbf{FPD} & \textbf{MLE} & \textbf{MAP} & \textbf{FPD}\\ 
\hline 
Simple Kriging & 0.88 & 0.92 & 0.94 & 25.9 & 29.3 & 31.1 \\
\hline
Ordinary Kriging & 0.87 & 0.91 & 0.93 & 25.7 & 28.5 & 30.4 \\
\hline
Affine Kriging & 0.90 & 0.91 & 0.92 & 26.0 & 26.6 & 27.3 \\
\hline
\end{tabular} 
\caption{Coverage and mean length of 95\% Prediction Intervals when emulating the 7-dimensional Rastrigin function augmented by a linear function (Matérn anisotropic geometric correlation kernel with smoothness $\nu=5/2$).}
\label{Tab:IP_Rastrigin_lin}
\end{center}
\end{table}

The addition of the linear function causes a decrease in performance for Prediction Intervals of both Simple and Ordinary Kriging, in the sense that coverage decreases while mean length increases for MAP and FPD. And the coverage of MLE sinks so much -- from 94\% to 88\% for Simple Kriging and from 91\% to 87\% for Ordinary Kriging -- that its performance may also be said to decrease, even though its mean length is slightly lower. \medskip

The performance of Affine Kriging is unchanged, however, whether one considers the MLE or MAP plug-in methods or the method using the full posterior distribution. This suggests that with a stronger linear component, Affine Kriging would be clearly preferable to Simple or Ordinary Kriging.

To test this, we 
emulate the 7-dimensional Rastrigin function, to which we add a stronger linear term: $(x_1,x_2,x_3,x_4,x_5,x_6,x_7) \mapsto 120 \sum_{i=1}^7 x_i$.

\begin{table}[!ht]
\begin{center}
\begin{tabular}{|c||c|c|c||c|c|c|}
\hline
\textsc{Rastrigin + $120 \sum_{i=1}^7 x_i$ }&  \multicolumn{3}{c||}{\textbf{Coverage}} & \multicolumn{3}{c|}{\textbf{Mean length}} \\
\hline
\textbf{Kriging model} & \textbf{MLE} & \textbf{MAP} & \textbf{FPD} & \textbf{MLE} & \textbf{MAP} & \textbf{FPD}\\ 
\hline 
Simple Kriging & 0.90 & 0.94 & 0.96 & 27.6 & 33.0 & 37.6 \\
\hline
Ordinary Kriging & 0.88 & 0.92 & 0.94 & 26.9 & 30.5 & 32.3 \\
\hline
Affine Kriging & 0.90 & 0.92 & 0.92 & 26.0 & 26.9 & 27.3 \\
\hline
\end{tabular} 
\caption{Coverage and mean length of 95\% Prediction Intervals when emulating the 7-dimensional Rastrigin function augmented by a linear function (Matérn anisotropic geometric correlation kernel with smoothness $\nu=5/2$).}
\label{Tab:IP_Rastrigin_lin_120}
\end{center}
\end{table}

For Simple Kriging, Prediction Intervals coverage and mean length are higher when $120 \sum_{i=1}^7 x_i$ is added to the Rastrigin function (Table \ref{Tab:IP_Rastrigin_lin_120}) rather than $100 \sum_{i=1}^7 x_i$ (Table \ref{Tab:IP_Rastrigin_lin}). This is also true, though to a lesser extent, of Ordinary Kriging. The performance of Affine Kriging, on the other hand, still remains the same because
it can account for any linear term by seeing it as part of the mean function. 
Simple Kriging (assuming the mean function to be null) and Ordinary Kriging do not have this luxury and must assume a greater variance for the Gaussian process, which results in 
more conservative Predictive Intervals. \medskip

Gathering the results obtained above, we conclude that Universal Kriging only significantly improves performance if the trend belongs to the assumed mean function space $\mathcal{F}_p$ and if it stands out. In other words, the signal/noise ratio must be high, where the signal is here the ``true'' mean function and the noise is the stationary Gaussian Process added to it. When no trend of the expected form can be discerned, like when emulating the Ackley or Rastrigin function through Affine Kriging, 
then there is no significant benefit to using Universal instead of Simple Kriging. When the ratio is high, as in the case of the Rastrigin function with the addition of the greater linear term $120 \sum_{i=1}^7 x_i$, Universal Kriging (if the mean function space $\mathcal{F}_p$ is adequately defined) improves upon Simple Kriging, which becomes overly conservative. Further, when the emulated function is particularly smooth, Simple Kriging becomes capable of capturing the trend to some extent even if the mean function is misspecified, thanks to the mechanics of Gaussian conditioning.

\section{Conclusion}

In this work, we provided an Objective Bayesian solution to the problem of taking into account parameter uncertainty when performing prediction based on a Universal Kriging model with anisotropic Matérn autocorrelation kernel. The reference posterior on the location parameter $\bs{\beta}$ and the variance parameter $\sigma^2$ is coupled with the Gibbs reference posterior on the vector of correlation lengths $\bs{\theta}$. By using the Gibbs reference posterior, which is the optimal compromise between the conditional reference posteriors on one correlation length $\theta_i$ based on the knowledge of all other correlation length $\theta_j$ ($j \neq i$), we bypass the problem of determining an ordering on the correlation lengths. Moreover, this solution allows for Gibbs sampling of the posterior distribution, which makes full-Bayesian inference or prediction tractable. \medskip

We proved that the Gibbs reference posterior exists and is proper in several Universal Kriging settings, depending on the number of available observation points and on the smoothness parameter of the Matérn kernel. \medskip

Numerical simulations show that Prediction Intervals produced by the full-Bayesian procedure based on the Gibbs reference posterior have better coverage than those produced by the Maximum Likelihood Estimator or even the Maximum A Posteriori estimator, and that their mean length is only moderately greater. \medskip

In addition, these simulations showed that when emulating deterministic functions,  there is no obvious advantage to using Universal Kriging over Simple Kriging, unless the trend strongly stands out and belongs to the assumed mean function space. \medskip

From a theoretical standpoint, the Universal Kriging setting poses specific problems when compared to the Simple Kriging setting.
As was shown (to our knowledge for the first time) by \cite{BDOS01}, the behavior of the integrated likelihood changes significantly depending on whether functions that take a non-null constant value on the design set are included in the mean function space $\mathcal{F}_p$. The integrated likelihood often fails to vanish in the neighborhood of perfect correlation in Ordinary Kriging models and \textit{a fortiori} in more complex Universal Kriging models where the constant term of the mean function is unknown. \cite{BDOS01} show in the isotropic framework that the reference prior adapts to this situation by being proper (at least for sufficiently rough correlation kernels -- their proof cannot be applied to kernels that are more than once differentiable). We were not able to prove the existence of the Gibbs reference posterior in such situations, however, which is why we require Assumption \ref{Hyp:vraisemblance_discriminante}. Although it is possible that closer analysis may allow us to relax this requirement, we find it more likely that Assumption \ref{Hyp:vraisemblance_discriminante} is the price we pay for defining the Gibbs reference posterior as a compromise between incompatible conditional reference posterior distributions. Indeed, each conditional maximizes the expected information of the model when all but one correlation length are fixed at finite values, i.e. in a context where \textit{perfect correlation is impossible}, whatever may be the value of the unfixed correlation length. Therefore, it is conceivable that in the absence of penalization by the integrated likelihood of the kind given by Assumption \ref{Hyp:vraisemblance_discriminante}, the conditionals may place too much weight on high values of the unfixed correlation length for the Gibbs reference posterior to be well defined. \medskip

Taking into account this restriction in Theorem \ref{Thm:compromis_gibbs_existe}, we proved that the Gibbs reference posterior exists and is the limit of a uniformly converging Markov Chain Monte-Carlo (MCMC) algorithm for commonly used Matérn anisotropic geometric and tensorized correlation kernels when the design set has enough points (cf. Propositions \ref{Prop:compromis_gibbs_existe_kriOrdin} and \ref{Prop:compromis_gibbs_existe_kriD1}). More generally, we would
conjecture that for any noninteger smoothness $\nu \in (1,+\infty)$, and if the mean function space $\mathcal{F}_p$ does not contain polynomials of degree higher than $[\nu]-1$, there exists some lower bound on the cardinal of the design set over which the Gibbs reference posterior exists and the MCMC algorithm uniformly converges to it. However, this lower bound may be too high for practical purposes. \medskip

Future work may involve gaining a better understanding of the significance of the Gibbs reference posterior as a compromise between the incompatible reference conditionals on correlation lengths. This method was primarily intended as a practical means of solving the problem of giving an objective posterior distribution on correlation lengths in the case of anisotropic correlation kernels, where the reference posterior is intractable and may not be proper. But its theoretical properties beyond its propriety, its invariance under reparametrizations of the type $f((\theta_1,...,\theta_r)\trans) = (f_1(\theta_1),...,f_r(\theta_r))\trans$ and its apparent good frequentist performances remain unknown.

\section*{Acknowledgments}
The author would like to thank his PhD advisor Professor Josselin Garnier (École Polytechnique, Centre de Mathématiques Appliquées) for his guidance, Loic Le Gratiet (EDF R\&D, Chatou) and Anne Dutfoy (EDF R\&D, Saclay) for their advice and helpful suggestions.
The author acknowledges the support of the French Agence Nationale de la Recherche (ANR), under grant ANR-13-MONU-0005 (project CHORUS).

\pagebreak

\begin{appendices}

\section{Matérn kernels} \label{App:noyaux_Matérn}

In this work, we 
use the following convention for the Fourier transform: the Fourier transform $\widehat{g}$ of a smooth function $g : \R^r \rightarrow \R$ verifies $ g(\bs{x}) = \int_{\R^r} \widehat{g} (\bs{\omega}) e^{i \langle \bs{\omega} | \bs{x} \rangle} d \bs{\omega} $ and $ \widehat{g} (\bs{\omega}) = (2 \pi)^{-r} \int_{\R^r} g(\bs{x}) e^{-i \langle \bs{\omega} | \bs{x} \rangle} d \bs{x} $. 

Let us set up a few notations.

\begin{enumerate}[(a)]
\item $\mathcal{K}_\nu$ is the modified Bessel function of second kind with parameter $\nu$ ;
\item $K_{r,\nu}$ is the $r$-dimensional Matérn isotropic covariance kernel with variance 1, correlation length 1 and smoothness $\nu \in (0,+\infty)$ and $\widehat{K}_{r,\nu}$ is its Fourier transform:

\begin{enumerate}[(i)]
\item $\forall \bs{x} \in \R^r$,

\begin{equation} \label{Eq:Matern_anis_geom}
K_{r,\nu}(\bs{x}) = \frac{1}{\Gamma(\nu) 2^{\nu - 1} } \left( 2 \sqrt{\nu} \|\bs{x}\| \right) ^{\nu} \mathcal{K}_\nu \left( 2 \sqrt{\nu} \|\bs{x}\| \right) \; ;
\end{equation}
\item 
$
\forall \bs{\omega} \in \R^r,
$
\begin{equation}
\widehat{K}_{r,\nu} (\bs{\omega}) 
= \frac{M_r(\nu)}{( \|\bs{\omega}\|^2 + 4 \nu  )^{  \nu + \frac{r}{2}  }}
\text{ with }
M_r(\nu) = \frac{\Gamma ( \nu + \frac{r}{2} ) (2 \sqrt{\nu} )^{2 \nu } }{ \pi^{\frac{r}{2} } \Gamma(\nu) }.
\end{equation}

\end{enumerate}

\item $K_{r,\nu}^{tens}$ is the $r$-dimensional Matérn tensorized covariance kernel with variance 1, correlation length 1 and smoothness $\nu \in \R_+$ and $\widehat{K}_{r,\nu}^{tens}$ is its Fourier transform:

\begin{enumerate}[(i)]
\item $\forall \bs{x} \in \R^r$,

\begin{equation} \label{Eq:Matern_tens}
K_{r,\nu}^{tens}(\bs{x}) = \prod_{j=1}^r K_{1,\nu}(\bs{x}_j) \; ;
\end{equation}
\item $\forall \bs{\omega} \in \R^r$,

\begin{equation}
\widehat{K}_{r,\nu}^{tens} (\bs{\omega}) =  \prod_{j=1}^r \widehat{K_{1,\nu}}(\bs{\omega}_j).
\end{equation}
\end{enumerate}

\item let us adopt the following convention: if $\bs{t} \in \R^r$, 
  $ \frac{\bs{t}}{\bs{\theta}} = \left( \frac{t_1}{\theta_1},...,\frac{t_r}{\theta_r} \right)$ and 
  $ \bs{t} \, \bs{\mu} = \left( t_1 \mu_1,..., t_r \mu_r \right)$. \medskip

\end{enumerate}

We define the Matérn geometric anisotropic covariance kernel with variance parameter $\sigma^2$, correlation lengths $\bs{\theta}$ (resp. inverse correlation lengths $\bs{\mu}$) and smoothness $\nu$ as the function $\bs{x} \mapsto \sigma^2 K_{r,\nu}\left(\frac{\bs{x}}{\bs{\theta}}\right)$ (resp. $\bs{x} \mapsto \sigma^2 K_{r,\nu}\left(\bs{x}\bs{\mu}\right)$).

Similarly, we define the Matérn tensorized covariance kernel with variance parameter $\sigma^2$, correlation lengths $\bs{\theta}$ (resp. inverse correlation lengths $\bs{\mu}$) and smoothness $\nu$ as the function $\bs{x} \mapsto \sigma^2 K_{r,\nu}^{tens}\left(\frac{\bs{x}}{\bs{\theta}}\right)$ (resp. $\bs{x} \mapsto \sigma^2 K_{r,\nu}^{tens} \left(\bs{x}\bs{\mu}\right)$). \medskip

\section{Proofs of section \ref{Sec:prior_gibbs_univ}} \label{Sec:prior_gibbs_univ:App}

\begin{proof}[Proof of Lemma  \ref{Lem:variance_matrix}]

\medskip

As $ \bs{M} $ is a symmetric matrix, the spectral theorem guarantees the existence of a diagonal matrix $ \bs{ \Lambda } $ and an orthogonal matrix $ \bs{O} $ such that $ \bs{M} = \bs{O} \trans \bs{ \Lambda } \bs{O} $, with the diagonal coefficients of $ \bs{ \Lambda } $ being the eigenvalues of $ \bs{M} $. Setting $ \bs{U}_0 := \bs{O} \bs{U} $, we can now compute $ \V [ \bs{U}_0 \trans \bs{ \Lambda } \bs{U}_0 ] = \V [ \bs{U} \trans \bs{M} \bs{U} ] $, $ \bs{U}_0 $ following the uniform distribution on $S^{n-1}$. \medskip

Let $ ( \lambda_i )_{ 1 \leqslant i \leqslant  n } $ be the eigenvalues of $ \bs{M} $. \medskip
We can write $ \V [ \bs{U}_0 \trans \bs{ \Lambda } \bs{U}_0 ] = \V [ \sum_{ 1 \leqslant i \leqslant  n } \lambda_i X_i ] $, where $ X_i $ ($ 1 \leqslant i \leqslant  n $) are nonnegative identically distributed random variables such that  $ \sum_{ 1 \leqslant i \leqslant  n } X_i = 1 $.

\begin{equation}
\label{VarianceSomme}
\begin{split}
\V \left[ \sum_{ i=1 }^{ n} \lambda_i X_i \right] 
&= \V [ X_1 ] \sum_{i=1}^{ n } \lambda_i^2 + 2 \Cov(X_1,X_2) \sum_{ 1 \leqslant i < j \leqslant n } \lambda_i \lambda_j.
\end{split}
\end{equation}

Obviously, $ \E [ X_1 ] = \frac{1}{ n } $ and thus 
$
\Cov( X_1 , X_2 ) 
= -1/ (n - 1) \V [ X_1 ].
$

\begin{equation} \label{VarianceCalcul}
\begin{split}
\frac{ \V \left[ \sum_{ i=1 }^{ n } \lambda_i X_i \right] }{ \V [ X_1 ] }
&= \sum_{i=1}^{ n } \lambda_i^2 - \frac{1}{ n - 1 } \sum_{ i = 1 }^n \lambda_i \sum_{ j \neq i } \lambda_j \\
&= \left( 1 + \frac{1}{ n - 1 } \right) \left( \Tr \left[ \bs{M}^2 \right] - \frac{1}{ n } \Tr \left[ \bs{M} \right]^2 \right). \\
\end{split}
\end{equation}

\end{proof}

\section{Proofs of the existence of the Gibbs reference posterior} \label{App:Gibbs_posterior}

The proof of the existence and uniqueness of the Gibbs reference posterior that was used in \citet{Mur18} to deal with the Simple Kriging setting is inadequate in the Universal Kriging setting because the projection $\bs{W} \trans$ may make key facts used in \citet{Mur18}  untrue. In the following, we provide replacements for the parts of the proof in \citet{Mur18} that are invalid in the Universal Kriging setting.

The proof contained two parts, one dealing with ``low correlations'', that is $\norme \to +\infty$ and one with ``high correlation'', that is $\norme \to 0$. \medskip

\subsection{Accounting for low correlation : \texorpdfstring{$\norme \to \infty$}{Lg}}

Concerning the part about $\norme \to +\infty$, we need to make sure that Corollary 33 of \citet{Mur18} remains true.

Define the functions $h_i$ by 
\begin{equation}
h_i( \mu_i \; | \; \manqueI ) :=  \sqrt{\Tr \left[ \left( \deIrivee \right)^2 \right] } = \left\| \deIrivee \right\|.
\end{equation}

The conclusion of Corollary 33 of \citet{Mur18} is that that there exist $S>0$ and $0<a<b$ such that, whenever $\norme \geqslant S$, 

\begin{equation} \label{Eq:encadrement_densite_prior_cond}
 a \; h_i( \mu_i \; | \; \manqueI ) \leqslant
f_i( \mu_i \; | \; \manqueI ) \leqslant
b \; h_i( \mu_i \; | \; \manqueI ).
\end{equation}

We need to find conditions under which this is true.
While the right inequality is obvious, the left inequality is harder to show. \medskip

Fix $\bs{\alpha} = \bs{\mu} / \norme_\infty $. Then define $\bs{L}_{i,\bs{\alpha}} = \lim \limits_{\norme \to \infty} \deIrivee / \left\| \deIrivee\right\|_\infty$. \medskip

We now give an explicit form for $\bs{L}_{i,\bs{\alpha}}$. Let $\bs{X}$ be the $n \times r$ matrix representing the design set, and let $\bs{X_\alpha}$ be the matrix $ \bs{X} Diag(\bs{\alpha})$, where $Diag(\bs{\alpha})$ is the $r \times r$ diagonal matrix whose diagonal is the vector $\bs{\alpha}$.

\begin{prop} \label{Prop:decrit_Lialpha_anisgeom}
If the Matérn kernel is anisotropic geometric, then $L_{i,\bs{\alpha}}$ is the symmetric $n \times n$ matrix with null diagonal whose nondiagonal coefficients are given by the following rule : its $(a,b)$ coefficient ($a,b \in [\![1,n]\!]$ and $a \neq b$) is $-1$ if the $a$-th and $b$-th point in the design set $\bs{X_\alpha}$ achieve minimal Euclidean distance within this design set, and $0$ otherwise.
\end{prop}

\begin{proof}
We only prove the result when $\nu>1$, but the proof is very similar in the case where $0<\nu\leqslant1$.

\citet{AS64} (formula 9.7.2.) yields that an equivalent for the one-dimensional Matérn kernel when $t \to +\infty$:
\begin{equation}
K_{1,\nu}(t) \sim \frac{\sqrt{\pi/2}}{\Gamma(\nu)2^{\nu-1}} (2\sqrt{\nu}t)^{\nu-1/2} \exp(-2\sqrt{\nu}t)
\end{equation}
From \citet{AS64} (formula 9.6.28.), we obtain that:
\begin{equation}
K_{1,\nu}'(t) = -\frac{2\nu t}{\nu-1} K_{1,\nu-1} \left(\sqrt{\frac{\nu}{\nu-1}} t \right)
\sim -2\sqrt{\nu} \frac{\sqrt{\pi/2}}{\Gamma(\nu)2^{\nu-1}} (2\sqrt{\nu}t)^{\nu-1/2} \exp(-2\sqrt{\nu}t)
\sim -2\sqrt{\nu} K_{1,\nu}(t)
\end{equation}

The result follows after recalling that $\frac{\partial}{\partial \mu_i} K_{r,\nu}(\bs{\mu x}) = \mu_i x_i^2 \| \bs{\mu x} \|^{-1} K_{1,\nu}'(\|\bs{\mu x}\|) $. When $\norme \to \infty$,
\begin{equation}
\frac{\partial}{\partial \mu_i} K_{r,\nu} (\bs{\mu x})
\sim -2\sqrt{\nu} \mu_i x_i^2 \| \bs{\mu x} \|^{-1} \frac{\sqrt{\pi/2}}{\Gamma(\nu)2^{\nu-1}} (2\sqrt{\nu} \| \bs{\mu x} \|)^{\nu-1/2} \exp(-2\sqrt{\nu} \| \bs{\mu x} \|).
\end{equation}

In the case where $0<\nu\leqslant1$, $\frac{\partial}{\partial \mu_i} K_{r,\nu}(\bs{\mu x})$ also has an equivalent when $\norme \to \infty$ whose prominent factor is $\exp(-2\sqrt{\nu} \| \bs{\mu x} \|)$, so the end result is the same.
\end{proof}

\begin{prop} \label{Prop:decrit_Lialpha_tens}
If the Matérn kernel is tensorized with smoothness $\nu>1$, then $L_{i,\bs{\alpha}}$ is the matrix with nonpositive coefficients such that $\left\| L_{i,\bs{\alpha}} \right\|_\infty = 1$ which is proportional to the symmetric matrix described hereafter : it has null diagonal and its nondiagonal coefficients are given by the following rule : its $(a,b)$ coefficient ($a,b \in [\![1,n]\!]$ and $a \neq b$) is $0$ if the $a$-th and $b$-th point in the design set $\bs{X_\alpha}$ do not achieve minimal 1-distance within this design set, and $ \alpha_i^{\nu - 1/2} \left| x_i^{(a)} - x_i^{(b)} \right|^{\nu + 1/2} \prod_{j \neq i} \alpha_j^{\nu - 1/2} \left|x_j^{(a)} - x_j^{(b)} \right|^{\nu - 1/2} $ if they do.
\end{prop}

\begin{rmq}
If the Matérn kernel is tensorized with smoothness $0<\nu \leqslant 1$, then the same rule applies but with different formula when minimal 1-distance is achieved.
\end{rmq}

\begin{proof}
The proof is similar to that of Proposition \ref{Prop:decrit_Lialpha_anisgeom}.
\end{proof}

\begin{cor} \label{Cor:rang_Lialpha}
For Matérn anisotropic geometric and tensorized kernels, if the design set $\bs{X}$ is randomly chosen according to the Uniform probability distribution on $(0,1)^{rn}$, then almost surely, whatever $i \in [\![1,n]\!]$ and $\bs{\alpha}$ in $\R^r$ such that $\| \bs{\alpha} \|_\infty = 1$, $L_{i,\bs{\alpha}}$ has rank lower or equal to $2r$.
\end{cor}

\begin{proof}
Almost surely, whatever $\bs{\alpha}$ in $\R^r$ such that $\| \bs{\alpha} \|_\infty = 1$, the design set $\bs{X_\alpha}$ has at most $r$ couples of distinct points achieving equal distance (whether that distance be the 1- or 2-distance). \textit{A fortiori}, it has at most $r$ couples of distinct points achieving minimal distance.
\end{proof}

With fixed $\bs{\alpha}$, as $\norme \to \infty$, we have

\begin{equation}
f_i( \mu_i \; | \; \manqueI ) \sim 
\left\| \deIrivee\right\|_\infty
\sqrt{ \Tr \left[ \left( \bs{W} \trans \bs{L}_{i,\bs{\alpha}} \bs{W} \right)^2 \right]
- \frac{1}{n-p} \Tr \left[  \bs{W} \trans \bs{L}_{i,\bs{\alpha}} \bs{W} \right]^2 
}
\end{equation}

We may recognize the factor under the square root as the variance (multiplied by $n-p$) of the eigenvalues (accounting for multiplicity) of the matrix $\bs{W} \trans \bs{L}_{i,\bs{\alpha}} \bs{W}$. If the premise of Corollary \ref{Cor:rang_Lialpha} holds, and if $2r < n-p$, then it is null if and only if $\bs{W} \trans \bs{L}_{i,\bs{\alpha}} \bs{W}$ is the null matrix. Assumption \ref{Hyp:moyenne_non_zero} is designed to prevent this from happening.

\begin{prop}
Assume $2r<n-p$. For Matérn anisotropic geometric or tensorized correlation kernels, if the design set $\bs{X}$ is randomly chosen according to the Uniform probability distribution on $(0,1)^{rn}$,
then almost surely,
Assumption \ref{Hyp:moyenne_non_zero} implies that
\begin{equation}
\min_{i \in [\![1,n]\!], \| \bs{\alpha} \|_\infty=1} \sqrt{ \Tr \left[ \left( \bs{W} \trans \bs{L}_{i,\bs{\alpha}} \bs{W} \right)^2 \right]
- \frac{1}{n-p} \Tr \left[  \bs{W} \trans \bs{L}_{i,\bs{\alpha}} \bs{W} \right]^2 
} > 0.
\end{equation}
\end{prop}

\begin{proof}
First, set $i \in [\![1,n]\!]$ and $\bs{\alpha}$ in $\R^r$ such that $\| \bs{\alpha} \|_\infty = 1$. We prove that $\bs{W} \trans \bs{L}_{i,\bs{\alpha}} \bs{W}$ is not the null matrix. \medskip

Assume that it is and that Assumption \ref{Hyp:moyenne_non_zero} holds. Assumption \ref{Hyp:moyenne_non_zero} implies that the intersection of the vector space spanned by $\bs{P}$ and the image of $\bs{L}_{i,\bs{\alpha}}$ is $\{\bs{0}_n\}$. Therefore, for any $\bs{z} \in \R^{n-p}$, if $\bs{L}_{i,\bs{\alpha}} \bs{W} z \neq \bs{0}_n$, then $\bs{W} \trans \bs{L}_{i,\bs{\alpha}} \bs{W} z \neq \bs{0}_{n-p}$, which contradicts the assumption that $\bs{W} \trans \bs{L}_{i,\bs{\alpha}} \bs{W}$ is the null matrix. So  $\bs{L}_{i,\bs{\alpha}} \bs{W}$ is the null $n \times (n-p)$ matrix, and thus the vector space spanned by $\bs{W}$ is included in the kernel of $\bs{L}_{i,\bs{\alpha}}$. This implies that $\bs{L}_{i,\bs{\alpha}} \bs{P} \bs{P} \trans = \bs{L}_{i,\bs{\alpha}}$, and then that $\bs{P} \bs{P} \trans \bs{L}_{i,\bs{\alpha}} = \bs{L}_{i,\bs{\alpha}}$. However, per Propositions \ref{Prop:decrit_Lialpha_anisgeom} and  \ref{Prop:decrit_Lialpha_tens}, all vectors in the image of $\bs{L}_{i,\bs{\alpha}}$ have at most $2r$ non-null elements when expressed in the canonical base of $\R^n$, so Assumption \ref{Hyp:moyenne_non_zero} implies that $\bs{P} \trans \bs{L}_{i,\bs{\alpha}}$ is the null $p \times n$ matrix, and thus that $\bs{L}_{i,\bs{\alpha}}$ is the null $n \times n$ matrix, which is untrue. \medskip

So, under Assumption \ref{Hyp:moyenne_non_zero}, whatever $i \in [\![1,n]\!]$ and $\bs{\alpha}$ in $\R^r$ such that $\| \bs{\alpha} \|_\infty = 1$, $\bs{W} \trans \bs{L}_{i,\bs{\alpha}} \bs{W}$ is not the null matrix and thus has a non-null eigenvalue. Moreover, $2r<n-p$ implies, according to Corollary \ref{Cor:rang_Lialpha}, that it also almost surely has a null eigenvalue, so the standard deviation of its eigenvalues is positive. As the number of possible matrices  $\bs{L}_{i,\bs{\alpha}}$ (with $i \in [\![1,n]\!]$ and $\bs{\alpha}$ in $\R^r$ such that $\| \bs{\alpha} \|_\infty = 1$) is almost surely finite, this yields the result.
\end{proof}

\begin{cor}
Assume $2r<n-p$. For Matérn anisotropic geometric or tensorized correlation kernels, if the design set $\bs{X}$ is randomly chosen according to the Uniform probability distribution on $(0,1)^{rn}$,
then almost surely,
Assumption \ref{Hyp:moyenne_non_zero} implies that there exist $S>0$ and $0<a<b$ such that Equation (\ref{Eq:encadrement_densite_prior_cond}) holds.
\end{cor}

\subsection{Accounting for high correlation : \texorpdfstring{$\norme \to 0$}{Lg}}

In the part of the proof in \citet{Mur18} concerning $\norme \to 0$, we used a the series expansion of $\coIrr$. This expansion may be heavily modified by premultiplication by $\bs{W} \trans$ and postmultiplication by $\bs{W}$. 

In the case where $\nu<1$, there is no material change unless the vector $\bs{1}$ belongs to the vector space spanned by $\bs{H}$. 

\begin{proof}[Proof of Proposition \ref{Prop:nu<1} ]
Because $\bs{1}$ does not belong to the vector space spanned by $\bs{H}$, $\bs{W} \trans \bs{11} \trans \bs{W}$ has rank 1 and so the proof of this result is the same as in the Simple Kriging case.
\end{proof}

If $\bs{1}$ \textit{does} belong to the vector space spanned by $\bs{H}$, further study would be needed to assess whether or not the above theorem still applies, essentially because we cannot count on $L(\bs{y}|\bs{\mu})$ vanishing as $\norme \to 0$. 

Let us now focus on the case where $\nu>1$. We reproduce key facts given by Lemma 27 and Proposition 42 of \citet{Mur18}:

\begin{lem}
For any Matérn anisotropic geometric or tensorized correlation kernel with smoothness parameter $\nu > 1$, if a coordinate-distinct design set is used, there exists $a>0$ such that when $\norme \to 0$:
\begin{enumerate}
\item $\left\| \deIrivee \right\| = O(\mu_i)$;
\item $\left\| \coIrr^{-1} \right\| = O(\norme^{-a})$.
\end{enumerate}
\end{lem}

\begin{cor} \label{Cor:prior<1}
There exists $S>0$ such that, for any $\bs{\mu} \in (0,+\infty)^r$ such that $\norme \leqslant S$ and $\mu_i \leqslant \norme^a$, $f_i(\mu_i | \manqueI) \leqslant 1$. 
\end{cor}

We combine the previous fact with a useful universal majoration of $f_i(\mu_i | \manqueI)$.

\begin{prop} \label{Prop:prior_majoration}
For an $r$-dimensional anisotropic geometric or tensorized Matérn correlation kernel with smoothness parameter $\nu$ pertaining to a design set containing $n$ coordinate-distinct points,
$\forall \bs{\mu} \in [0,+\infty)^r$ such that $\mu_i > 0$,

\begin{equation}
f_i(\mu_i | \bs{\mu}_{-i})
\leqslant (n-p) (2\nu+r) \mu_i^{-1}
\end{equation}
\end{prop}

\begin{proof}

Whatever $x,y \in \R$, $K_{1,\nu} (x-y) = \int_{\R} \widehat{K}_{1,\nu} ( \omega ) e^{ i \omega (x-y) } d \omega $.

For the sake of concision, we only consider the case where the Matérn kernel is anisotropic geometric, as the changes in the case of a tensorized kernel are straightforward. \medskip

Moreover, we start by proving the result in the case where $\bs{W}$ is the identity matrix $\bs{I}_n$ (Simple Kriging case).

\begin{equation}
\sum_{j,k=1}^n \xi_j \xi_k  K_{r,\nu} \left(\left(\bs{x}^{(j)} - \bs{x}^{(k)}\right) \bs{\mu} \right) 
= \int_{\R^r} \widehat{K_{r,\nu}} (\bs{\omega}) \left| \sum_{j=1}^n \xi_j e^{i \omega_i  \mu_i x_i^{(j)} + i \left\langle \bs{\omega}_{-i} \left| \bs{\mu}_{-i} \bs{x}_{-i}^{(j)}\right. \right\rangle }  \right|^2 d\bs{\omega} 
= M_r(\nu) \mu_i^{-1} I_{\bs{\mu}} (\bs{\xi}) \\
\end{equation}

where

\begin{align}
M_r(\nu) &= \frac{\Gamma ( \nu + \frac{r}{2} ) (2 \sqrt{\nu} )^{2 \nu } }{ \pi^{\frac{r}{2} } \Gamma(\nu) } \\
I_{\bs{\mu}} (\bs{\xi}) &= \int_{\R^r} \left( 4 \nu + \mu_i^{-2} s_i^2 +  \left\| \frac{ \bs{s}_{-i} }{ \bs{\mu}_{-i} } \right\|^2 \right)^{-\frac{r}{2} - \nu} \left| \sum_{j=1}^n \xi_j e^{i \langle \left. \bs{s} \right| \bs{x}^{(j)} \rangle } \right|^2 d \bs{s} \\
\end{align}

We also have

\begin{equation}
\begin{split}
\frac{d}{d\mu_i} \sum_{j,k=1}^n \xi_j \xi_k  K_{r,\nu} \left(\left(\bs{x}^{(j)} - \bs{x}^{(k)}\right) \bs{\mu} \right) 
&= - M_r(\nu) \mu_i^{-1} I_{\bs{\mu}} (\bs{\xi}) + M_r(\nu) \mu_i^{-1} \frac{d}{d\mu_i}  I_{\bs{\mu}} (\bs{\xi})
\end{split}
\end{equation}

\begin{equation}
\begin{split}
\frac{d}{d\mu_i} I_{\bs{\mu}} (\bs{\xi}) &= 2\left(\frac{r}{2} + \nu\right) \mu_i^{-3}
\int_{\R^r} s_i^2 \left( 4 \nu + \mu_i^{-2} s_i^2 +  \left\| \frac{ \bs{s}_{-i} }{ \bs{\mu}_{-i} } \right\|^2 \right)^{-\frac{r}{2} - \nu - 1} \left| \sum_{j=1}^n \xi_j e^{i \langle \left. \bs{s} \right| \bs{x}^{(j)} \rangle } \right|^2 d \bs{s} \\
&= (2\nu + r) \mu_i^{-3}
\int_{\R^r} \frac{s_i^2}{ 4 \nu + \mu_i^{-2} s_i^2 +  \left\| \frac{ \bs{s}_{-i} }{ \bs{\mu}_{-i} } \right\|^2} \left( 4 \nu + \mu_i^{-2} s_i^2 +  \left\| \frac{ \bs{s}_{-i} }{ \bs{\mu}_{-i} } \right\|^2 \right)^{-\frac{r}{2} - \nu} \left| \sum_{j=1}^n \xi_j e^{i \langle \left. \bs{s} \right| \bs{x}^{(j)} \rangle } \right|^2 d \bs{s} \\
\end{split}
\end{equation}

From this, we obtain that for any non-null vector $\bs{\xi} \in \R^n$,

\begin{equation}
0 < \frac{d}{d\mu_i} I_{\bs{\mu}} (\bs{\xi})
\leqslant (2\nu + r) \mu_i^{-1} I_{\bs{\mu}} (\bs{\xi})
\end{equation}

Now let us define the matrix $\bs{F}_{\bs{\mu}}$ as the matrix representing in the canonical base of $\R^n$ the positive definite quadratic form $\bs{\xi} \mapsto M_r(\nu) \mu_i^{-1} \frac{d}{d\mu_i}  I_{\bs{\mu}} (\bs{\xi})$. From the previous calculations, we gather that $\frac{d}{d\mu_i} \coIrr = - \mu_i^{-1} \coIrr + \bs{F}_{\bs{\mu}}$. This in turn yields $\left( \deIrivee \right) \coIrr^{-1} = - \mu_i^{-1}\bs{I}_n + \bs{F}_{\bs{\mu}} \coIrr^{-1}$ and 
$\left(\left( \deIrivee \right) \coIrr^{-1} \right)^2 = \mu_i^{-2} \bs{I}_n + \left(\bs{F}_{\bs{\mu}} \coIrr^{-1}\right)^2 - 2 \mu_i^{-1} \bs{F}_{\bs{\mu}} \coIrr^{-1}$.

\begin{align}
\Tr \left[ \left( \deIrivee \right) \coIrr^{-1} \right] &= -n \mu_i^{-1} + \Tr \left[ \bs{F}_{\bs{\mu}} \coIrr^{-1} \right] \\
\Tr \left[ \left(\left( \deIrivee \right) \coIrr^{-1}\right)^2 \right] &= n \mu_i^{-2} + \Tr \left[\left(\bs{F}_{\bs{\mu}} \coIrr^{-1}\right)^2 \right] - 2 \mu_i^{-1} \Tr \left[ \bs{F}_{\bs{\mu}} \coIrr^{-1} \right] \\
\Tr \left[ \left(\left( \deIrivee \right) \coIrr^{-1}\right)^2 \right] - \frac{1}{n} \Tr \left[ \left( \deIrivee \right) \coIrr^{-1} \right]^2
&= \Tr \left[\left(\bs{F}_{\bs{\mu}} \coIrr^{-1}\right)^2\right] - \frac{1}{n} \Tr \left[ \bs{F}_{\bs{\mu}} \coIrr^{-1} \right]^2
\end{align}

$\bs{F}_{\bs{\mu}}$ and $\coIrr^{-1}$ being two symmetric positive definite matrices, their product $\bs{F}_{\bs{\mu}} \coIrr^{-1}$ is diagonalizable and all its eigenvalues are positive. Thus $\Tr \left[\left(\bs{F}_{\bs{\mu}} \coIrr^{-1}\right)^2\right] \leqslant \Tr \left[ \bs{F}_{\bs{\mu}} \coIrr^{-1} \right]^2$. \medskip

Let $(\bs{\xi}_{\bs{\mu}}^j)_{1 \leqslant j \leqslant n}$ be a basis of unit eigenvectors of $\coIrr^{-1}$. Then 

\begin{equation}
\Tr \left[ \bs{F}_{\bs{\mu}} \coIrr^{-1} \right]
= \sum_{j=1}^n \left(\bs{\xi}_{\bs{\mu}}^j\right) \trans \bs{F}_{\bs{\mu}} \coIrr^{-1} \bs{\xi}_{\bs{\mu}}^j
= \sum_{j=1}^n \frac{\left(\bs{\xi}_{\bs{\mu}}^j\right) \trans \bs{F}_{\bs{\mu}} \bs{\xi}_{\bs{\mu}}^j} {\left(\bs{\xi}_{\bs{\mu}}^j\right) \trans \coIrr \bs{\xi}_{\bs{\mu}}^j}
\leqslant n(2\nu + r) \mu_i^{-1}
\end{equation} 

This implies that 
\begin{align}
\Tr \left[ \left(\left( \deIrivee \right) \coIrr^{-1}\right)^2 \right] - \frac{1}{n} \Tr \left[ \left( \deIrivee \right) \coIrr^{-1} \right]^2
&\leqslant n(n-1)(2\nu+r)^2 \mu_i^{-2} \\
\sqrt{ \Tr \left[ \left(\left( \deIrivee \right) \coIrr^{-1}\right)^2 \right] - \frac{1}{n} \Tr \left[ \left( \deIrivee \right) \coIrr^{-1} \right]^2 }
&\leqslant n(2\nu+r) \mu_i^{-1}
\end{align}

Now, if $\bs{W}$ is not the identity matrix, then the previous proof still holds, albeit with  some alterations.
Instead of considering all non-null vectors $\bs{\xi} \in \R^n$, we consider only those which can be expressed as $\bs{W} \bs{\xi_W}$, with $\bs{\xi_W}$ belonging to $\R^{n-p}$. In the same vein, once it comes to computing $\Tr \left[ \bs{W} \trans \bs{F}_{\bs{\mu}} \bs{W} \left( \bs{W} \trans \coIrr \bs{W} \right)^{-1} \right]$, we use a basis $(\bs{\xi}_{\bs{W,\mu}}^j)_{1 \leqslant j \leqslant n-p}$ of unit eigenvectors of $\left( \bs{W} \trans \coIrr \bs{W} \right)^{-1}$.

\end{proof}

\begin{prop}
With a Matérn anisotropic geometric or tensorized correlation kernel with smoothness $\nu>1$, if a design set with coordinate-distinct points is used, then Assumption \ref{Hyp:vraisemblance_discriminante} implies that there exists $\epsilon'>0$ such that $L(\bs{y} | \bs{\mu}) f_i(\mu_i | \manqueI) = O(\mu_i^{-1 + \epsilon'})$ when $\norme \to 0$.
\end{prop}

\begin{proof}
Assumption \ref{Hyp:vraisemblance_discriminante} ensures that $L(\bs{y} | \bs{\mu})$ is bounded as a function of $\bs{\mu}$.
Because of Corollary \ref{Cor:prior<1}, and using said Corollary's notations, we know that there exists $M>0$ such that, for any $\bs{\mu} \in (0,+\infty)^r$ such that $\mu_i \leqslant \norme^a$, 
$L(\bs{y} | \bs{\mu}) f_i(\mu_i | \manqueI) \leqslant M$. \medskip

Let us now focus on the $\bs{\mu} \in (0,+\infty)^r$ such that $\mu_i \geqslant \norme^a$. Then $\norme^\epsilon \leqslant \mu_i^{\epsilon/a}$. Choosing $\epsilon' = \epsilon / a$, combining Assumption \ref{Hyp:vraisemblance_discriminante} and Proposition \ref{Prop:prior_majoration} yields the result.
\end{proof}

Using essentially the proof of Proposition 52 of \citet{Mur18}, we obtain the following result.

\begin{prop} \label{Prop:aux_compromis_gibbs_existe}
In a Universal Kriging model with a Matérn anisotropic geometric or tensorized correlation kernel with smoothness $\nu>1$, if a design set with coordinate-distinct points is used, then Assumption \ref{Hyp:vraisemblance_discriminante} implies that the conditional posterior distribution $\pi_i(\mu_i | \bs{y}, \manqueI)$, seen as a function of $\bs{\mu}$, is continuous over $\{ \bs{\mu} \in [0,+\infty)^r : \mu_i \neq 0 \}$.
\end{prop}

\begin{proof} [Proof of Theorem \ref{Thm:compromis_gibbs_existe}]
With the help of Proposition \ref{Prop:aux_compromis_gibbs_existe}, this proof is similar to the proof of Proposition 12 in \citet{Mur18}.
\end{proof}

Consider the following set of conditions :

\begin{enumerate}
\item $1<\nu<2$ and $n>r+1$;
\item $2<\nu<3$ and $n>(r+1)(r/2 + 2)$.
\end{enumerate}

\begin{prop} \label{Prop:krigeage_ordinaire}
In the case of Ordinary Kriging, under the conditions of Theorem \ref{Thm:compromis_gibbs_existe}, if one of the previous conditions is satisfied, then there exists a hyperplane $\mathcal{H}$ of $\R^n$ such that, provided $\bs{y} \in \R^n \setminus \mathcal{H}$, Assumption \ref{Hyp:vraisemblance_discriminante} is true.
\end{prop}

\begin{prop} \label{Prop:krigeage_affine}
In the case of Universal Kriging where the mean function space is included within the space of polynomials of degree 0 or 1, if $2<\nu<3$ and $n>r(r+1)/2 + 2r + 3$, then there exists a hyperplane $\mathcal{H}$ of $\R^n$ such that, provided $\bs{y} \in \R^n \setminus \mathcal{H}$, Assumption \ref{Hyp:vraisemblance_discriminante} is true.
\end{prop}

The proofs of both previous propositions are similar to the proofs of Lemmas 50 and 51 in \citet{Mur18}. \medskip

\begin{proof}[Proof of Propositions \ref{Prop:compromis_gibbs_existe_kriOrdin} and \ref{Prop:compromis_gibbs_existe_kriD1}]
Propositions \ref{Prop:compromis_gibbs_existe_kriOrdin} and \ref{Prop:compromis_gibbs_existe_kriD1} are obtained by combining Theorem \ref{Thm:compromis_gibbs_existe} with Propositions \ref{Prop:krigeage_ordinaire} and \ref{Prop:krigeage_affine} respectively.
\end{proof}

\end{appendices}

\pagebreak
\bibliography{biblio}
\end{document}